\documentclass[11pt,a4paper]{article}
\usepackage[T1]{fontenc}
\usepackage[utf8]{inputenc}
\usepackage{authblk}

\usepackage{amssymb,amsthm,amsmath}
\usepackage[numbers,sort&compress]{natbib}
\usepackage{color}
\usepackage{graphicx}
\usepackage{subfigure}
\usepackage{tikz}
\usepackage[colorlinks,
linkcolor=red,
anchorcolor=blue,
citecolor=green
]{hyperref}

\let\newpf\proof \let\proof\relax 
\newenvironment{proof}{\newpf[\proofname]}{\qed\endtrivlist}

\newcommand{\ba}{\overline{A}}

\def\be{\begin{equation}}
\def\ee{\end{equation}}

\def\ba{{\begin{align}}}
\def\ea{{\end{align}}}

\def\bm{\begin{matrix}}
\def\em{\end{matrix}}

\def\0{{\mathbf 0}}

\def\cal{\mathcal}

\newtheorem{Theorem}{Theorem}[section]
\newtheorem{Lemma}{Lemma}[section]
\newtheorem{Proposition}{Proposition}[section]
\newtheorem{Corollary}{Corollary}[section]
\newtheorem{Remark}{Remark}[section]
\newtheorem{Example}{Example}[section]
\newtheorem{Definition}{Definition}[section]

\numberwithin{equation}{section}

\theoremstyle{definition}

\newcommand{\C}{{\mathbb C}}

\newcommand{\R}{{\mathbb R}}

\newcommand{\Z}{{\mathbb Z}}

\def\B0{{\bold{0}}}

\catcode`\@=12

\def\Empty{}
\newcommand\oplabel[1]{
  \def\OpArg{#1} \ifx \OpArg\Empty {} \else
    \label{#1}
  \fi}

\newcommand{\comm}[1]{}
\newcommand{\comment}[1]{}

\title{Global structure of the spectrum of periodic Non-hermitian  Jacobi operators}
\author[1]{Hui Lu}

\author[2]{Jiangong You}

\affil[1]{School of Mathematics, Nanjing Audit University, Nanjin  {\rm 211815}, China}
\affil[2]{Chern Institute of Mathematics and LPMC, Nankai University, Tianjin \rm{300071}, China}

\date{} % 去掉日期

\begin{document}
\maketitle
\abstract{The global structure of the spectrum of periodic non-Hermitian Jacobi operators is described by the discriminant and its stationary points. We also give necessary and sufficient conditions for real spectrum and single interval spectrum.}

\section{Introduction}

Non-Hermitian operators
\begin{equation}\label{1.1}
	(Ju)_n=a_{n}u_{n+1}+b_nu_n+c_{n-1}u_{n-1}, \ \ n\in \Z,
\end{equation}
acting on $l^2(\Z)$ are mathematical models for open quantum systems.  The periodic  non-Hermitian Jacobi operator $J$, 
with periodic  $\{a_n\},\{b_n\}$, $\{c_n\}$, that is 
$a_{n+N}=a_n,b_{n+N}=b_n,c_{n+N}=c_n,N\in \Z^+$, usually be used as an  approximation  of complex models, such as random and quasi-perioidc models, by physists.

The resolvent set $\rho(J)$ consists of the complex numbers $\lambda$  for which $\lambda  Id-J$ is one-to-one and onto, where $Id$ is the identity. The spectrum of $J$ is $\sigma(J)=\C \backslash \rho(J)$. 

Let
$$
J_{-1 }=\left(\begin{array}{cccc}
	\  \ &  	\  \ & 	\  \ & 	\  \   \\
	&  & & \\
	&  &  &  \\
	a_N& &  &
\end{array}\right), \ \
J_{0 }=\left(\begin{array}{cccc}
	b_{1} & a_1 &  &  \\
	c_1 & b_2 & \ddots & \\
	&  \ddots & \ddots &a_{N-1}\\
	&  & c_{N-1} & b_{N}
\end{array}\right),\ \
J_{1 }=\left(\begin{array}{cccc}
	\ \ &  	\ \ & 	\ \ &  c_N \\
	&  & & \\
	&  &  &  \\
	& &  &
\end{array}\right).$$
These three matrices are all square matrices of order $N$, and all blank positions correspond to 0. 
Then  $J$ can be equivalently written as the  doubly infinite Laurent operator
$$J=\left(\begin{array}{ccccccc}
	\ddots& \ddots &\ddots &  & & \\
	& J_1 & J_0&J_{-1} & &  \\
	&  & J_1 & [J_0]&J_{-1} &   \\
	&  &  &J_1 & J_0&J_{-1}  \\
	& & & & \ddots& \ddots &\ddots
\end{array}\right),$$ 
where $[J_0]$ denotes the $(0,0)$ entry. Denote by  $J(\theta)$  the  $N$-order square matrix 
$$J(\theta )=J_{0 }+J_{1 }e^{i\theta }+J_{-1 }e^{-i\theta },  \ \  \theta \in(-\pi,\pi].$$  
G. K. Kumar and S. H. Kulkarni  (\cite{GS}) proved that 
$$\sigma(J)= \bigcup_{\theta \in(-\pi,\pi]}\sigma(J(\theta ))=\bigcup_{\theta \in(-\pi,\pi]}\left\{ \lambda \in \C \mid \det( \lambda Id-J(\theta ))=0 \right\}.$$
This result can also be directly obtained from the standard Floquet transform (\cite{KP}). 

Direct computation leads to 
$$\det( \lambda Id-J(\theta ))=P(\lambda)-ae^{-i\theta }-ce^{i\theta },$$ 
where  $a= \Pi_{i=1}^Na_i, c=  \Pi_{i=1}^Nc_i $, $P(\lambda)$ is a monic polynomial of degree $N$ independent of $\theta$ which is called  the discriminant of $J$.
Thus,
$$\sigma(J(\theta ))=\left\{ \lambda \in \C \mid P(\lambda)=ae^{-i\theta }+ce^{i\theta } \right\}=P^{-1}(ae^{-i\theta }+ce^{i\theta }) , $$
$$\sigma(J)=\left\{ \lambda \in \C \mid P(\lambda)\in \{ae^{-i\theta }+ce^{i\theta }\mid \theta \in(-\pi,\pi] \} \right\}=P^{-1}(\mathcal{E})  ,$$
where $\mathcal{E} =\{ae^{-i\theta }+ce^{i\theta }\mid \theta \in(-\pi,\pi] \}$.  

Any monic complex polynomial $P$ is the discriminant of some periodic Jacobi operators (may not be unique).
This is proved by Papanicolaou \cite{VGP}, where he considered the periodic operator  (\ref{1.1}) with $a_n=c_n,n=1,2,\cdots,N$. His proof is also valid for $a_n\ne c_n$.

The set $\mathcal{E}$ is an ellipse if  $|a|\neq |c|$, 
$\cal{E}$  is degenerated to  a  line segment if $|a|=|c|\neq 0,  $ especially $\cal{E}=[-2,2]$ for Schr\"odinger operators. \footnote{Proof: Let $a=r_ae^{i\varphi_a}, c=r_ce^{i\varphi_c }$ are the corresponding polar coordinates, with  $\varphi_a,\varphi_c\in(-\pi,\pi]$. Let  
	$\varphi=\frac{\varphi_a+\varphi_c}{2}$, $\widehat{\theta } =\theta-\frac{\varphi_a-\varphi_c}{2}$, then 
	\begin{align*}
		\mathcal{E}&=   \{r_ae^{i(-\theta +\varphi_a)}+r_c e^{i(\theta +\varphi_c)}\mid \theta  \in (-\pi,\pi] \}=  e^{i\varphi} \cdot \{r_ae^{i(-\theta +\varphi_a-\varphi)}+r_c e^{i(\theta +\varphi_c-\varphi)}\mid \theta  \in (-\pi,\pi] \}\\
		&= e^{i\varphi}  \cdot  \{r_ae^{-i\widehat{\theta }}+r_c e^{i\widehat{\theta }}\mid \widehat{\theta } \in (-\pi-\frac{\varphi_a-\varphi_c}{2},\pi-\frac{\varphi_a-\varphi_c}{2}]\}.
	\end{align*}
	If $r_a\neq r_c,\mathcal{E}$ is an ellipse. If $r_a=r_c \neq 0 ,\cal{E}$ is a line segment. 
	If $r_a=r_c = 0 ,\cal{E}=\{0\}.$ 
}

$\cal{E}$  is further degenerated to $\{0\}$ if $a=c=0$. In this case, $\sigma(J)=P^{-1}(\mathcal{E}) =P^{-1}(0)$,  so $\sigma(J)  $ is composed of $N$  zeros (counting the multiplicity) of $P$, which are independent of $\theta$. For all $\theta\in (-\pi,\pi]$, these $N$  zeros are the eigenvalues of $J(\theta )$. If $(u_1(\theta ),u_2(\theta ),...,u_N(\theta ))^T$ is an eigenvector of $J(\theta )$ corresponding to $\lambda$, then  there is a solution to $Ju=\lambda u$ obeying $u_{n+N}(\theta )=e^{-i\theta }u_n(\theta )$ for all $n\in \Z$. 
In particular, if $a_{n_0}=c_{n_0}=0$, then $a_{{n_0}+kN}=c_{{n_0}+kN}=0$ for all $k\in \Z$, and $J$ splits as  $J=\bigoplus_{k \in \mathbb{Z}} A$, where $A$ is a $N$-order matrix. 
For example, if ${n_0}=N$, then $A=J(0 )$. So  the case $a=c=0$ is an easy case and 
we suppose that $a$ and $c$ are not both $0$ from now on. 

It is known that the spectrum of non-Hermitian operators are much complicated than Hermitian operators, (\cite{MIS, FSR, VAT1, HH, ES,KP,BS,VA,NVZ,LV,VAT2,YL,JV,VAT3,GT,OM,CMB,FV,BS,GS,VGP0,VGP}).
It has been proved that the spectrum of periodic Jacobi operators, denoted by $\sigma(J)$, is a union of finite many pieces of analytic curves, which may have finite many intersections(\cite{OK,OV,VGP}). There are  also local descriptions for $\sigma(J)$(\cite{MK,KCS,OV}).

In this paper, we will give a full description of   $\sigma(J)$ by the discriminant $P$, especially by its stationary points
\begin{equation*}\label{stationary point set}
	\cal{S}=\{\lambda  \in \C\ \  |\ P'(\lambda)=0\}.
\end{equation*}
We say a stationary point $\lambda$ is of order $k$ if $P^{(j)}(\lambda)=0, j=1, \cdots, k$ and $P^{(k+1)}(\lambda)\ne 0$. 
We denote by  $\tau (\lambda)$ the order of $\lambda$. For $ \lambda\notin\cal{S}$, we set $\tau (\lambda)=0$ for simplicity.   For convinience, we denote   by 
$$m_P(E,F):=\sum_{\lambda \in E\cap  P^{-1}(F)} \tau(\lambda) $$
the number of the stationary points in $E\cap P^{-1}(F)$ counting the multiplicity  where  $E,F\subset \C$.\footnote{It is obvious that $m_P(E,F)=m_P(\cal{S} \cap E,F)=m_P(E,P(\cal{S} )\cap F)\in \{0,1,2,\cdots,N-1\}$. 
	Moreover, 
	\begin{align*}
		&m_P(E_1 \cap  E_2,F)=m_P(E_1,F )-m_P(E_1 \backslash E_2,F),\    \ m_P(E,F_1\cap F_2)=m_P(E,F_1 )-m_P(E,F_1 \backslash F_2),\\
		&m_P(E_1 \cup  E_2,F)=m_P(E_1,F )+m_P(E_2\backslash E_1,F),\    \ m_P(E,F_1 \cup  F_2)=m_P(E,F_1 )+m_P(E,F_2 \backslash F_1).
	\end{align*}
	Especially, $m_P(\C,F )+m_P(\C,\C \backslash F)=m_P(\C,\C )=\sum_{\lambda\in 	 \C}\tau(\lambda)=\sum_{\lambda\in 	\cal{S}}\tau(\lambda)=N-1$.}

We will prove that the  global structure of $\sigma(J)$ is completely determined by  the  location of  $\cal{S}$ and $P(\cal{S})$.

Let $\gamma=\{z(t)|t\in I\}$ be a simple curve, where $I\subset \R$ is an interval.
We first give some notations.
\begin{Definition}
	A contiuous curve $\gamma_*=\{\lambda(t)|P(\lambda(t))=z(t) \text{ for all }t\in I\}$ contained in $P^{-1}(\gamma)$ is called a band if $P:\gamma_* \rightarrow \gamma$ is one-to-one. 
\end{Definition}
\begin{Remark}
	It is easy to see that each band $\gamma_*$ is a simple curve. If $\gamma$ is not a closed curve, then $P:\gamma_* \rightarrow \gamma$ is homeomorphic. If $\gamma$ is a closed curve, then  $P:\gamma_* \rightarrow \gamma$ is homeomorphic if and only if  $\gamma_*$ is  a closed curve. 
\end{Remark}

\begin{Definition}
	A simple closed curve contained in $P^{-1}(\gamma)$ is called a petal, and the number of petals in $P^{-1}(\gamma)$ is denoted by $\cal{C}_P(\gamma)$. 
	The union of petals sharing a common point is called a flower, the common point is called the center of the flower.
\end{Definition}
\begin{Remark}
	A flower contains at least two petals.	We don't call a single petal a flower.
\end{Remark}

\begin{Definition}
	The maximum  connected subset in $P^{-1}(\gamma)$ is called a  bouguet, and the number of bouguets in $P^{-1}(\gamma)$ is denoted by $\cal{B}_P(\gamma)$.  
\end{Definition}

More precisely, we get the following
\begin{Theorem}\label{main1}
	Suppose that  $|a|\ne |c|$. Let $W$ be the closed domain\footnote{The domain $W$ is closed, let $\mathring{W}$ be the interior of $W$ and  $\partial W$ be the boundary of $W$.} bounded by $\cal E$, then
	
	{\rm(1)} $\sigma(J)=P^{-1}(\mathcal{E}) $  is a bounded closed set, and $\C \backslash P^{-1}(W)$ is connected.
	
	{\rm(2)} {\rm Petal decomposition:}  
	$\sigma(J)$ has a unique decomposition of petals with
	$$\cal{C}_P(\mathcal{E})=1+ m_P(\C,\C \backslash \mathring{W} ). 
	$$
	
	{\rm(3)} {\rm Band decomposition:}  
	There is a unique band decomposition $\sigma(J)=\bigcup^N_{n=1}{\gamma_n }$, such that each petal is composed of bands end to end. Each band  is  piecewise analytic and the no smooth points are in $\cal{S}$.\footnote{If $\lambda$ is a piecewise point of a band, then  there must be other bands pass through $\lambda$.}
	Moreover, 
	let $C$ be a petal  in  $ \sigma(J)$, and $D$ be the domain  bounded by $C$. Then $P(\mathring{D})=\mathring{W}$, and for any $\lambda_0\in \mathring{D}$ the number of bands  contained in $C$ is
	$$\frac{1}{2\pi i}\int_{C} \frac{P'(\lambda)}{P(\lambda)-P(\lambda_0)}d\lambda=1+m_P(\mathring{D},\C).$$
	
	{\rm(4)} {\rm Flower:} The number of flowers  in $ \sigma(J)$ is $\#(\cal{S}\cap \sigma(J)).$ $\lambda$ is a center of a flower if and only if $\lambda\in \cal{S}\cap  \sigma(J)$. If $\lambda$ is a center of a flower, there are $\tau(\lambda)+1$ petals in this flower, and the small neighborhood of $\lambda$ is divided into $2\tau(\lambda)+2$  regions by these $\tau(\lambda)+1$ petals,  each region has an angle   $\frac{\pi}{\tau(\lambda)+1}$.
	
	{\rm(5)} {\rm Bouquet decomposition:} $\sigma(J)$ has a unique bouquet decomposition with
	$$\cal{B}_P(\mathcal{E})=1+ m_P(\C,\C\backslash W). $$
	Morevover, each bouquet is either a single petal or composed of  flowers. 
	Let $\Gamma$ be a bouquet in $ \sigma(J)$, and $\Omega$ be the domain  bounded by $\Gamma$. Then, the number of flowers in $\Gamma$ is $\#(\cal{S}\cap \Gamma),$ the number of petals  contained in $\Gamma$ is 
	$1+m(\Gamma,\C), $ and the number of bands  contained in $\Gamma$ is 
	$1+m(\Omega,\C).$
\end{Theorem}	
\begin{Corollary}
	Generically, $\sigma(J)$ is composed of single petals\footnote{Because, generically $|a|\ne |c|$ and $m_P(\C,\mathcal{E})=0$.} with 
	$$	\cal{C}_P(\mathcal{E})=\cal{B}_P(\mathcal{E})=1+ m_P(\C,\C\backslash W). $$
\end{Corollary}

In Theorem \ref{main1}, we don't require $N$ to be the minimum positive period. Suppose $J$ is an operator with minimum positive periodic $N$, then $kN,k\in\Z^+$ is also a periodic of $J$. Let $P_{k}$ be the discriminant, $\mathcal{E}_k =\{a^ke^{-i\theta }+c^ke^{i\theta }\mid \theta \in(-\pi,\pi] \}$ and  $W_k$ be the domain bounded by $\cal E_k$.  Then, we have 
\begin{Corollary}\label{invariant1}
	$\sigma(J)=P_k^{-1}(\mathcal{E}_k) $, for all $k\in\Z^+$. In particular, $m_{P_k}(\C,\C \backslash \mathring{W} _k)$ and  $m_{P_k}(\C,\C \backslash W_k)$ are  independent of $k$.\footnote{Because the number of petals and bouguets in $\sigma(J)$ are  independent of $k$.}
\end{Corollary}
By Theorem \ref{main1}, the global structure of $\sigma(J)$ depends on the location of $\cal{S}$ and $P(\cal{S})$. For example, we directly have the following corollaries for  $|a|\ne |c|$.
\begin{Corollary}\label{bands intersect}
	At most one intersection between any two bands (or petals).\footnote{Because $\C \backslash P^{-1}(W)$ is connected.}
\end{Corollary}
\begin{Corollary}\label{one bouquet}
	$\sigma(J)$ has only one bouquet $\Leftrightarrow \cal{B}_P(\mathcal{E})=1\Leftrightarrow m_P(\C,\C\backslash W)=0\Leftrightarrow m_P(\C, W)=N-1\Leftrightarrow P(\cal{S})\subset W$.
\end{Corollary}
\begin{Corollary}\label{one petal}
	$\sigma(J)$ has only one petal \footnote{ This petal is formed by $N$ bands.} $\Leftrightarrow \cal{C}_P(\mathcal{E})=1\Leftrightarrow m_P(\C, \C \backslash \mathring{W})=0\Leftrightarrow m_P(\C, \mathring{W})=N-1\Leftrightarrow P(\cal{S})\subset \mathring{W}$.
\end{Corollary}
\begin{Corollary}\label{one band}
	A petal $C$ is composed of one band $\Leftrightarrow P(C)=\cal E$ is one-to-one $\Leftrightarrow\frac{1}{2\pi i}\int_{C} \frac{P'(\lambda)}{P(\lambda)-P(\lambda_0)}d\lambda=1$, for any $\lambda_0\in \mathring{D}$  $\Leftrightarrow P(\mathring{D})=\mathring{W}$ is one-to-one $\Leftrightarrow m_P(\mathring{D},\C)=0 $  $\Leftrightarrow \cal{S}\cap \mathring{D}=\emptyset$.
\end{Corollary}

\begin{Corollary}\label{n petals}
	$\sigma(J)$ consists of $N$ petals $\Leftrightarrow \cal{C}_P(\mathcal{E})=N \Leftrightarrow m_P(\C, \C \backslash \mathring{W})=N-1 \Leftrightarrow m_P(\C, \mathring{W})=0 \Leftrightarrow P(\cal{S})\cap  \mathring{W} =\emptyset$.
\end{Corollary}

\begin{Corollary}\label{n disjoint petals}
	$\sigma(J)$ consists of $N$ disjoint petals $\Leftrightarrow  \cal{B}_P(\mathcal{E})=N \Leftrightarrow m_P(\C,  \C \backslash W)=N-1 \Leftrightarrow m_P(\C, W)=0 \Leftrightarrow P(\cal{S})\cap  W=\emptyset$.
\end{Corollary}

By Theorem \ref{main1}, we can also construct operators whose spectrum satisfies some global requirements. 
\begin{Example}\label{Example N=5}
	We give an illustrating  example with   one bouquet, $N$ petals, and a flower with $N-1$ petals in 
	$\sigma(J)$.
	
	We can construct it in three steps. First, take  $\cal{S}=\{\lambda_1,\lambda_2\}$ and  $P'(\lambda)=N(\lambda-\lambda_1)^{N-2}(\lambda-\lambda_2)$. Second, take $a_1,a_2,\cdots,a_N,c_1,c_2,\cdots,c_N$, get $\cal{E}$, such that $|a|\neq |c|$ and $P(\cal{S})\subset \cal{E}$. Again, take $b_1,b_2,\cdots,b_N$, such that the discriminant is $P$.
	
	For example, let $N=5,\cal{S}=\{0,1\}$ and $P'(\lambda)=5\lambda^{3}(\lambda-1)$.
	Let $a_1=a_2=a_3=a_4=1,a_5=\frac{1}{8},c_1=c_2=\cdots=c_5=0$, then $a=\frac{1}{8},c=0$, $\mathcal{E} =\{\frac{1}{8}e^{-i\theta }\mid \theta \in(-\pi,\pi] \}$ is a circle. To satisfy $P(\cal{S})\subset \cal{E}$,  we need $P(\lambda)=\lambda^5-\frac{5}{4}\lambda^4+\frac{1}{8}$. By direct calculation, we have
	$P(\lambda)=(\lambda-b_1)(\lambda-b_2)(\lambda-b_3)(\lambda-b_4)(\lambda-b_5)$.  Let $b_1,b_2,\cdots,b_5$ be the zeros of $\lambda^5-\frac{5}{4}\lambda^4+\frac{1}{8}$, then we get a $5$-periodic Jacobi operator $J$, which meets the previous requirements. The corresponding spectrum is shown in Figure \ref{Picture of Example 1.1}(b).

	In the previous construction, if we fix $b_1,b_2,\cdots,b_5$ as the zeros of $\lambda^5-\frac{5}{4}\lambda^4+\frac{1}{8},c_1=\cdots=c_5=0$, and
arbitrarily take $a_1,a_2,\cdots,a_5$,  let $a= \Pi_{i=1}^5a_i$ be the parameter. Then  $\mathcal{E} =\{ae^{-i\theta }\mid \theta \in(-\pi,\pi] \}$ and $P(\lambda)=\lambda^5-\frac{5}{4}\lambda^4 + \frac{1}{8}$.

If $|a|>\frac{1}{8},P(\cal{S})\subset \mathring{W}$, by Corollary \ref{one petal}, $\sigma(J)$ has only one petal which is formed by  $5$ bands (Figure \ref{Picture of Example 1.1}(c)). 

If $|a|=\frac{1}{8},P(\cal{S})\subset \cal{E}=\partial W$ and $\tau(0)=3$,  by Corollary \ref{one bouquet},Corollary \ref{n petals} and Theorem \ref{main1}(4), these operators all meet the previous requirements, and they have the same spectrum (Figure \ref{Picture of Example 1.1}(b)).  

If $0<|a|<\frac{1}{8}$,\footnote{If $|a|=0,\sigma(J)=P^{-1}(0)=\{b_1,b_2,\cdots,b_5\}$. } $P(\cal{S})\cap  W=\emptyset$, by Corollary \ref{n disjoint petals}, $\sigma(J)$ consists of $5$ disjoint petals (Figure \ref{Picture of Example 1.1}(a)).

	\begin{figure}
		\centering
		\subfigure[$a=0.12$]{
			\begin{minipage}[b]{0.3\textwidth}
				\includegraphics[width=1\textwidth]{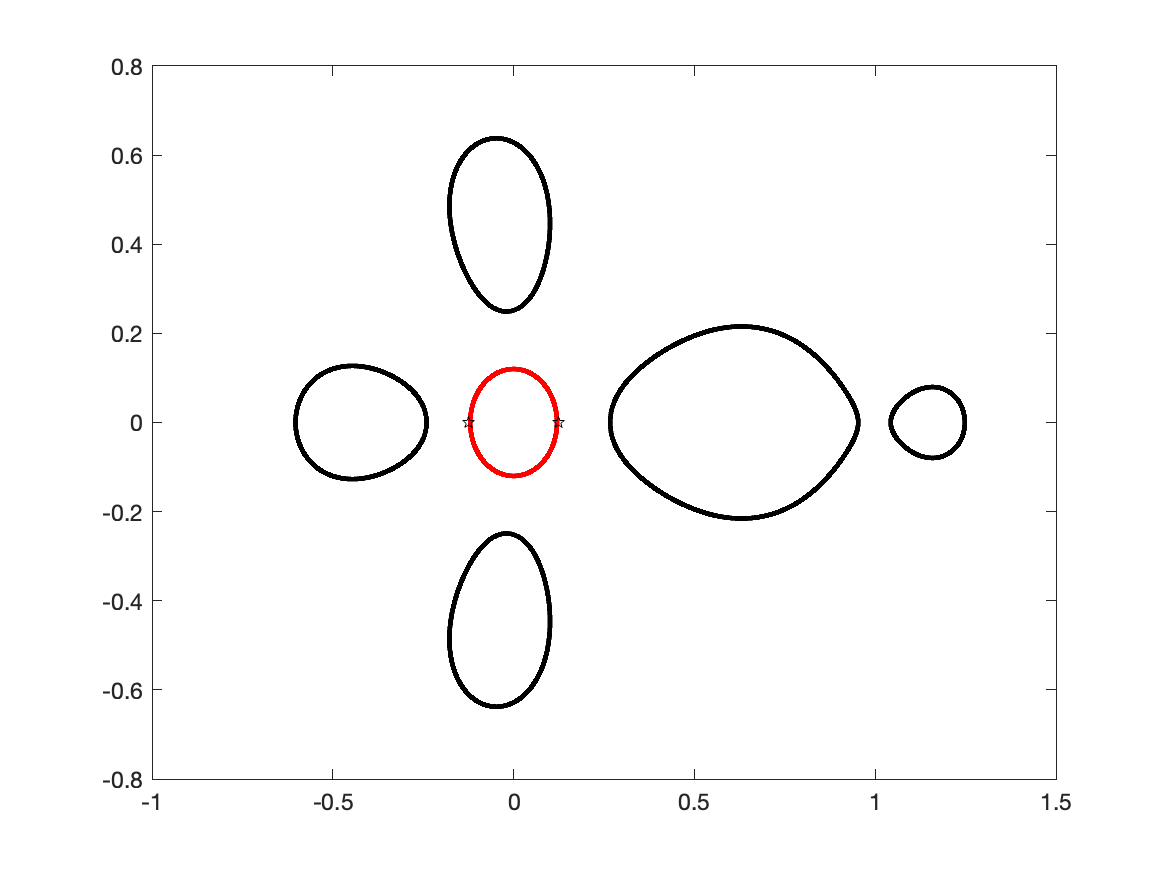}
			\end{minipage}
			%	\label{fig:hor_2figs_1cap_2subcap_1}
		}
		\subfigure[$a=0.125$]{
			\begin{minipage}[b]{0.3\textwidth}
				\includegraphics[width=1\textwidth]{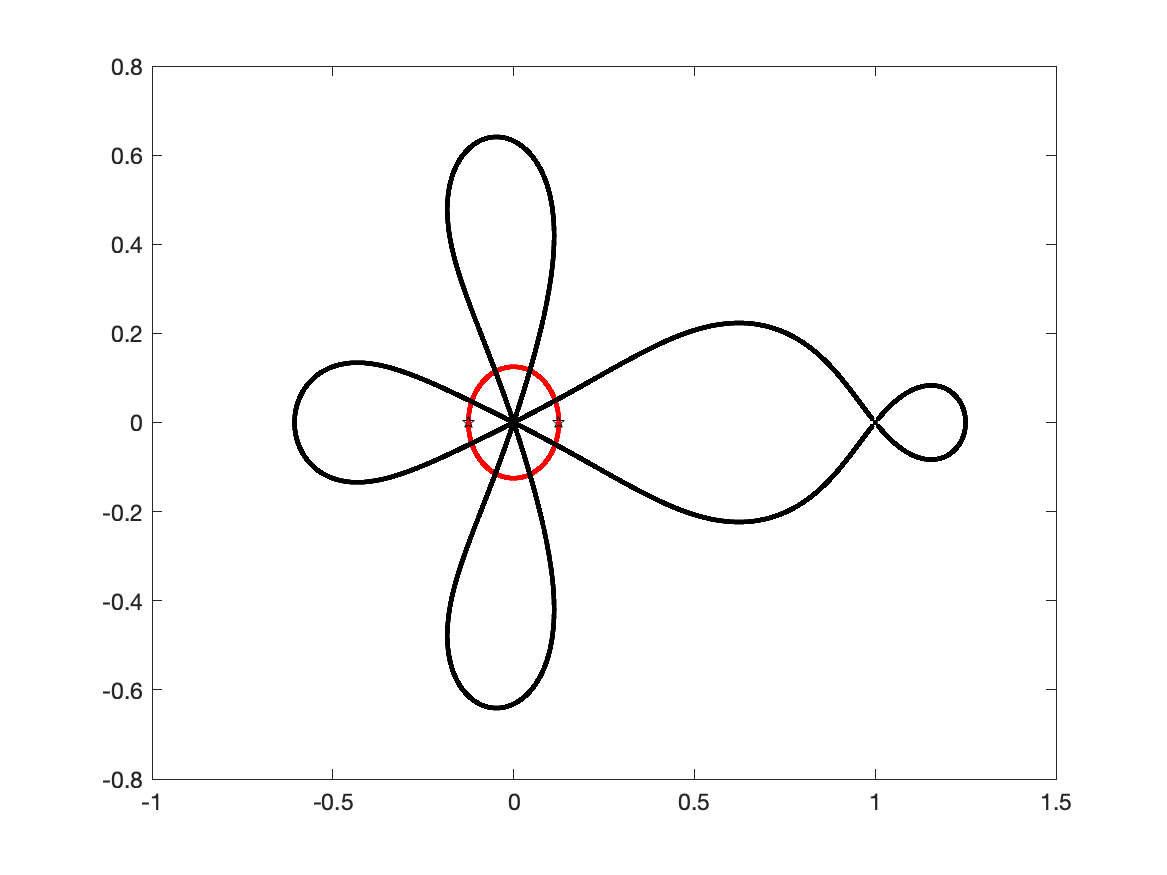}
			\end{minipage}
			%	\label{fig:hor_2figs_1cap_2subcap_2}
		}
		\subfigure[$a=0.13$]{
			\begin{minipage}[b]{0.3\textwidth}
				\includegraphics[width=1\textwidth]{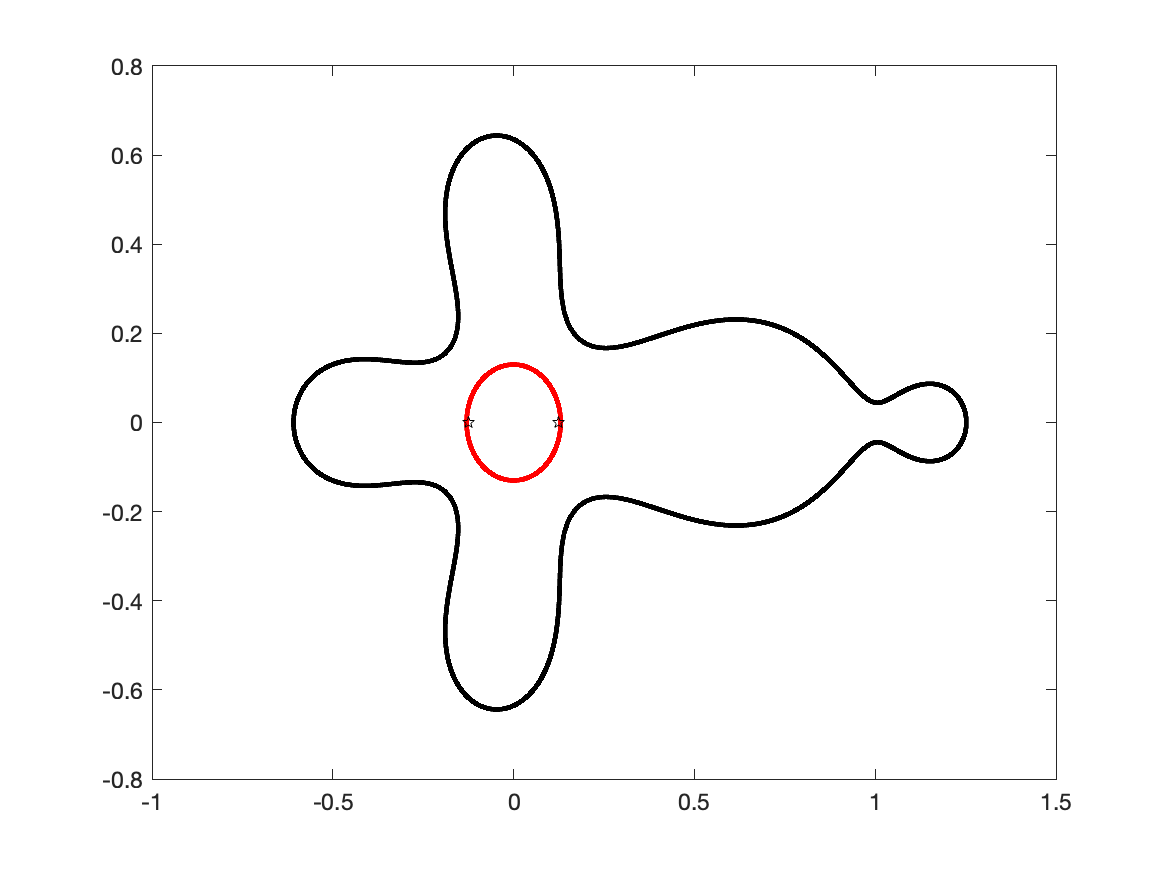}
			\end{minipage}
			%	\label{fig:hor_2figs_1cap_2subcap_2}
		}
		\caption{Example 1.1. The black curves is $\sigma(J)$, the red circle is $\cal{E}$, and the two five-pointed stars are $P(0),P(1)$.}
		\label{Picture of Example 1.1}
	\end{figure}

\end{Example}

Now we consider the case	$|a|= |c|\neq 0$,  which includes the Schr\"odinger operators. In this case,  $\cal{E}=\{e^{i\varphi}\cdot t|t\in [-2R,2R] \}=e^{i\varphi}  \cdot  [-2R,2R]$  is a line segment passing through the origin, where $R=|a|$, $\varphi=\frac{\varphi_a+\varphi_c}{2}$.  

\begin{Theorem}\label{main2}
	Suppose that  $|a|= |c|\neq 0$, then
	
	{\rm(1)} $\sigma(J)=P^{-1}(\mathcal{E}) $  is a bounded closed set, and $\C \backslash P^{-1}(\mathcal{E})$ is connected.

	{\rm(2)} {\rm Band decomposition:}  
	There is a band decomposition $\sigma(J)=\bigcup^N_{n=1}{\gamma_n }$, which is unique if and only if $P(\cal{S})\cap \mathcal{E}\backslash  \{\pm2Re^{i\varphi}\} =\emptyset$.\footnote{In particular, if $J$ is an Hermitian operator, then the band decomposition is unique. } Each band is  piecewise analytic and the no smooth points are in $\cal{S}$.

	{\rm(3)} $ \lambda$ is a intersection of  bands if and only if $\lambda\in \cal{S}\cap  \sigma(J)$. 
	Moreover, if $\lambda$ is   an endpoint of a band, then the small  neighborhood of $\lambda$ is divided into $\tau(\lambda)+1$  regions by these $\tau(\lambda)+1$ bands, each region has an angle $\frac{2\pi}{\tau(\lambda)+1}$. If $\lambda\in \sigma(J)$ is not an endpoint of any band, then the small  neighborhood of $\lambda$ is divided into $2\tau(\lambda)+2$  regions by these $\tau(\lambda)+1$ bands, each region has an angle $\frac{\pi}{\tau(\lambda)+1}$.

	{\rm(4)} {\rm Bouquet decomposition:} $\sigma(J)$ has a unique bouquet decomposition with 
	$$\cal{B}_P(\mathcal{E})=1+ m_P(\C,\C\backslash \mathcal{E})$$
	and  the number of bands contained in bouquet  $\Gamma$ is 
	$1+m(\Gamma,\C).$ 
	\end{Theorem}

\begin{Example}\label{example local doco}
	Let $N=2,a_1=a_2=c_1=c_2=1,b_1=i,b_2=-i$, then $\cal{E}=[-2,2]$ and $P=\lambda^2-1$. So 
	$\sigma(J)=P^{-1}([-2,2])=[-\sqrt{3},\sqrt{3}]\cup i\cdot [-1,1]$.  There are two different way of band decomposition, one is $\gamma_1=[-\sqrt{3},0]\cup i\cdot [-1,0],\gamma_2=[0,\sqrt{3}]\cup i\cdot [0,1]$, the other is $\gamma_1=[-\sqrt{3},0]\cup i\cdot [0,1],\gamma_2=[0,\sqrt{3}]\cup i\cdot [-1,0]$.
\end{Example}

We have the following corollaries for $|a|= |c|\neq 0$.
\begin{Corollary}\label{invariant2}
	$\sigma(J)=P_k^{-1}(\mathcal{E}_k) $,\footnote{Here $P_k$ and $\mathcal{E}_k$ are the same as in Corollary \ref{invariant1}.} for all $k\in\Z^+$. In particular, $m_{P_k}(\C,\C \backslash \mathcal{E} _k)$ is independent of $k$. 
\end{Corollary}
\begin{Corollary}\label{bands intersect+}
	There is no closed curve in $\sigma(J)$ and any two bands has at most one intersection.
\end{Corollary}
\begin{Corollary}
	$\sigma(J)$ consists of $N$ disjoint bands $\Leftrightarrow \cal{B}_P(\mathcal{E})=N \Leftrightarrow m_P(\C,  \C\backslash \mathcal{E})=N-1 \Leftrightarrow m_P(\C, \mathcal{E})=0 \Leftrightarrow P(\cal{S})\cap  \mathcal{E}=\emptyset$. \footnote{ For $|a|= |c|\neq 0$, generically $\sigma(J)$ consists of $N$ disjoint bands.}
\end{Corollary}
\begin{Corollary}
	$\sigma(J)$ has only one bouquet $\Leftrightarrow \cal{B}_P(\mathcal{E})=1 \Leftrightarrow \\
	m_P(\C,\C\backslash \mathcal{E})=0 \Leftrightarrow m_P(\C, \mathcal{E})=N-1 \Leftrightarrow P(\cal{S})\subset \mathcal{E}$.
\end{Corollary}

\begin{Corollary}\label{N bands end to end}
	$ \sigma(J)$ is composed of $N$ bands end to end $\Leftrightarrow \#\cal{S}=N-1$ and $P(\cal{S})\subset \{\pm 2Re^{i\varphi} \} $.
\end{Corollary}
\begin{Example}\label{unp}
	Consider the discrete free Laplacian $J=-\Delta$, i.e., 	$a_n=c_n=-1$ and $b_n=0$, for all $n\in \Z$. The corresponding $N$-periodic operator is denoted as ${J}_N$. 
	
	Let $\lambda=z+z^{-1}$, inductive knowable the discriminant	
	\begin{equation}\label{PN}
		{P}_N(\lambda)=z^N+z^{-N}=\left(\frac{\lambda+\sqrt{\lambda^2-4}}{2}\right)^N+\left(\frac{\lambda-\sqrt{\lambda^2-4}}{2}\right)^N,
	\end{equation}
	which is a monic real polynomial of degree $N$.  
	
	Let ${z}_n=e^{i \frac{n\pi}{N}}, n=0,1,2,\cdots,N$,  then ${\lambda}_{n}=2cos(\frac{n\pi}{N})$. By direct calculation, we have  $\cal{S}=\{{\lambda}_1,{\lambda}_2,\cdots,{\lambda}_{N-1}\},{P}_N(\cal{S})\subset \{\pm 2\} $, and ${P}_N({\lambda}_{n})=(-1)^n\cdot 2$, for $n=0,1,2,\cdots,N$. For $\mathcal{E}=[-2,2]$, by Corollary \ref{N bands end to end}, $\sigma({J}_N)$ is composed of $N$ bands end to end. Moreover, by Corollary \ref{invariant2}, $\sigma({J}_N)=\sigma({J}_1)={P}_1^{-1}([-2,2])=[-2,2],$
	hence  $$\sigma({J}_N)=\bigcup^N_{n=1}{\gamma_n }=\bigcup^N_{n=1}{[{\lambda}_n,{\lambda}_{n-1}] }.$$
\end{Example}
Finding the criterion for real spectrum of non-self-adjoint\footnote{An Hermitian operator is the physicist's version of an object that mathematicians call a self-adjoint operator.}  is one of the main issue of this subject(\cite{MFC}). In this paper, as applications of Theorem \ref{main2}, we give a necessary and sufficient condition for real spectrum (Theorem \ref{real spectrum}), as well as necessary and sufficient conditions for a single interval spectrum (Theorem \ref{interval spectrum2} and \ref{interval spectrum3}).

\begin{Theorem}\label{real spectrum}
	If $a$ and $c$ are not both $0$, then  	$ \sigma(J)\subset \R$ if and only if the following hold
	
	{\rm(1)} $a=\bar{c}$.
	
	{\rm(2)} $\cal{S}=\{\lambda_1,\lambda_2,\cdots,\lambda_{N-1}\}\subset \R$, with $\lambda_1>\lambda_2>\cdots>\lambda_{N-1}$.
	
	{\rm(3)}  $(-1)^nP(\lambda_n) \geq 2R, n=1,2,\cdots,N-1$.
\end{Theorem}	
\begin{Theorem}\label{interval spectrum2}  
	$ \sigma(J)=[\alpha,\beta]$, if and only if  the following hold
	
	{\rm(1)} $a=\bar{c}$.	
	
	{\rm(2)}  $\#\cal{S}=N-1,P(\cal{S})\subset \{\pm  2R\}$.
	
	{\rm(3)}	$P( \{\alpha, \beta\})\subset \{\pm  2R\}$, where $\alpha<\beta$ and $\alpha, \beta\notin \cal{S}$. 
\end{Theorem}
\begin{Theorem}\label{interval spectrum3}
	$ \sigma(J)=[\alpha,\beta]$, if and only if  
	$\cal{E}=[-2R,2R]$ and 
	$P(\lambda)=(\frac{\beta-\alpha}{4})^N\cdot{P}_N( \frac{4}{\beta-\alpha}\cdot (\lambda -\frac{\alpha+\beta}{2})),$ 
	where $\alpha<\beta,R=(\frac{\beta-\alpha}{4})^N$ and $P_N$ is defined in (\ref{PN}).
\end{Theorem}
More detailed results will be given in  Section 3.

The paper is arranged as the following: In Section 2,  we will discuss the global properties of inverse image of a curve under $P$, then as applications, we give the proof of Theorem \ref{main1} and Theorem \ref{main2}. In Section 3, as applications of Theorem \ref{main2}, we give necessary and sufficient conditions for $ \sigma(J)\subset \R$ and  for $ \sigma(J)=[\alpha, \beta]$  respectively,  and we will also  discuss the case of $ \sigma(J)$ belongs to a straight line in the complex plane $\C $.

\section{Proof of Theorem \ref{main1} and Theorem \ref{main2}} 
\subsection{The inverse image of a curve under polynomial}

In this section, we first discuss the global structure of $P^{-1}(\gamma) $ for a fixed  monic complex polynomial $P$ of degree $N$ and a piecewise smooth simple curve $\gamma$.  The most  interesting cases in applications are when $\gamma$ is an ellipse or a line segment. We first give a local description for $P^{-1}(\gamma)$.

\begin{Proposition}\label{local1}
	Let $\gamma=\{z=z(t) | t\in (t_0-\delta, t_0+\delta)\}$ be a smooth simple arc, with $z_0=z(t_0)$ and $\lambda_0 \in P^{-1}(z_0)$. Then we can choose a
	suffciently small  neighborhood  $U$ of $z_0$  such that $P(\cal{S}) \cap U \backslash \{z_0\}=\emptyset$. We denote by  $\hat{V}$ the maximum connected domain  in $P^{-1}(U)$  containing $\lambda_0$. 
	Then $P^{-1}(\gamma)\cap  \hat{V}$  consists of $2(\tau(\lambda_0)+1)$ smooth arcs that have a common  $\lambda_0$, and  these arcs divide $\hat{V}$
	into $2(\tau(\lambda_0)+1)$  regions. Moreover, adjacent arcs meet at $\lambda_0$ with  angle $\frac{\pi}{\tau(\lambda_0)+1}$. Let $U_1$ and  $U_2$  be two parts of $U$ separated by $\gamma$. Then the images of  these regions under $P$ are $U_1$ and  $U_2$ respectively.  Specifically, if $\tau(\lambda_0)=0$, i.e., $\lambda_0$ is not a stationary point of $P$, then $P^{-1}(\gamma)$ near $\lambda_0$ is a  smooth arc.  
	
\end{Proposition}
\begin{proof}
	If $\tau(\lambda_0)=0$, by the Implicit Function Theorem, $P^{-1}(\gamma)$ near $\lambda_0$ is a  smooth arc.

	Let $\lambda \in \hat{V}\cap P^{-1}(\gamma)$ and 
	$\theta =\arg(z'(t_0))$.
	Then
	$P(\lambda) \rightarrow P(\lambda_0)$, when $\lambda\rightarrow \lambda_{0}$,
	and $$\lim_{\lambda\rightarrow \lambda_{0}}\arg(P(\lambda)-P(\lambda_0))=\lim_{t\rightarrow t_{0}}\arg(z(t)-z(t_0))=\begin{cases}\theta,& t< t_{0} \\ \theta +\pi,& t>t_{0}. \end{cases} 
	$$
	For $\lambda_0 \in P^{-1}(z_0)$, we let $k=\tau(\lambda_0)+1$, then
	$$	P(\lambda)-P(\lambda_0)=\frac{P^{(k)}( \lambda _{0})  }{k!}(\lambda-\lambda_0)^k+O(\lambda-\lambda_0)^{k+1},$$ as $\lambda\rightarrow \lambda_{0}$.
	Thus, there exists $n\in \Z$, such that
	\begin{align*}\arg(P(\lambda)-P(\lambda_0))+2n\pi&=\arg(\frac{P^{(k)}( \lambda _{0})  }{k!}(\lambda-\lambda_0)^k+O(\lambda-\lambda_0)^{k+1})\\
		&=\arg(P^{(k)}( \lambda _{0}))+k\arg((\lambda-\lambda_0)+\arg(1+O(\lambda-\lambda_0)),
	\end{align*}
	as $\lambda\rightarrow \lambda_{0}$. Taking its limit, we get
	\begin{align*}
		\lim_{\lambda\rightarrow \lambda_{0}}\arg((\lambda-\lambda_0)&=\frac{	\arg(P(\lambda)-P(\lambda_0))-\arg(P^{(k)}( \lambda _{0}) )+2n\pi}{k}\\
		&=\begin{cases}\frac{\theta -\arg(P^{(k)}( \lambda _{0}) )+2n\pi}{k},& t< t_{0} \\ \frac{\theta -\arg(P^{(k)}( \lambda _{0}) )+(2n+1)\pi}{k},& t>t_{0}. \end{cases}
	\end{align*}
	This means that $P^{-1}(\gamma)$ near $\lambda_0$ consists of $2k$ smooth arcs with a common endpoint $\lambda_0$. Moreover, adjacent arcs meet at $\lambda_0$ with an angle $\frac{\pi}{k}$,  and the images of  these arcs under $P$ belong to  $\gamma$ with $t>t_0$  and  $t<t_0$, respectively. 
	
	Because $P(\cal{S}) \cap U\backslash \{z_0\}=\emptyset$, so $S\cap \hat{V} \backslash \{\lambda_0\}=\emptyset$. Thus, these $2k$  arcs only intersect at $\lambda _{0}$ in $\hat{V}$ and 
	divide $\hat{V}$ into $2k$  regions. The images of  these regions under $P$ are $U_1$ and  $U_2$, respectively. 
\end{proof}

\begin{Remark}\label{analytic}      
	If $\gamma $ is an analytic arc and $\tau(\lambda_0)=0$, then $P^{-1}(\gamma)$ near $\lambda_0$ is an analytic arc. 
\end{Remark}
\begin{Remark}\label{local1+}
	If $\gamma $ is a piecewise smooth simple arc, with angle $\vartheta $ at the piecewise point $z_0$. Then other conclusions in the Proposition \ref{local1} are still valid, except that those arcs meet at $\lambda_0$ with angle $\frac{\vartheta }{\tau(\lambda_0)+1}$ and $\frac{2\pi-\vartheta }{\tau(\lambda_0)+1}$, respectively. 
\end{Remark}
\begin{Remark}\label{local decom}
	For any $z_0\in \gamma$ and $\lambda_0 \in P^{-1}(z_0)$, we let $U_1 $ be the left neighborhood of $z_0$ along the increasing direction of $t$.
	Then, there are $\tau(\lambda_0)+1$ regions in Proposition \ref{local1} such that $P$ is homomorphic of those rigions to  $U_1$. The boundary curves of these $\tau(\lambda_0)+1$ regions in $\hat{V}$ are exactly the $2(\tau(\lambda_0)+1)$ arcs given in Proposition \ref{local1}, and the boundary of each region in $\hat{V}$ is a band of $P^{-1}(\gamma\cap U)$, i.e.,  it is one-to-one by $P$ to $\gamma\cap U$. 
	
	If $\tau(\lambda_0)>0$, there are other ways to define the bands.  For example, one can similarly  use the right neighborhood to define bands.

\end{Remark}

\begin{Corollary}\label{band deco1}
	Let $\gamma=\{z(t)|t\in I\}$ be a piecewise smooth simple arc, where $I\subset \R$ is an interval, then there is a  band decomposition $P^{-1}(\gamma)=\bigcup^N_{n=1}{\gamma_n }$. %All bands are piecewise smooth simple curve.
\end{Corollary}
\begin{proof}
	For any $t\in I$, equation $P(\lambda)=P(z(t))$ has $N$ roots $\lambda_n(t),n=1,2,\cdots,N$, counting the multiplicity.  Let $U_1(z(t))$ be the left neighborhood along the increasing direction of $t$. 
	By Remark \ref{local1+} and Remark \ref{local decom}, the local trend of the $N$ curves is completely determined. Thus we obtain a band decomposition $P^{-1}(\gamma)=\bigcup^N_{n=1}{\gamma_n}$, with $\gamma_n=\{\lambda_n(t)|t\in I\}$. 
\end{proof}

\begin{Remark}\label{not unique}
	If $\cal{S}\cap P^{-1}(\gamma)=\emptyset$, all $N$ bands will be uniquely defined and disjoint.
	If there is $\lambda\in \cal{S}\cap P^{-1}(\gamma)$, such that $P(\lambda)$ is not an endpoint of $\gamma$, then $P^{-1}(\gamma)$ contains  $2(\tau(\lambda)+1)$ arcs  in a small neighborhood of $\lambda$ and $\lambda$ is the unique common point of them. By Remark \ref{local decom}, the way of grouping two of them into bands   is not unique, so the band decomposition method is no longer unique (See Example \ref{example local doco}).
\end{Remark}

\begin{Corollary}\label{global0}
	Let $\cal{W}\subset \C$ be a simply connected bounded closed set, then $P^{-1}(\cal{W})$ is a bounded closed set, and the set $\C \backslash P^{-1}(\cal{W})$ is connected.
\end{Corollary}	
\begin{proof}
	$P^{-1}(\cal{W})$ is obviously a bounded closed set, because $\cal{W}\subset  \C$ is a bounded closed set and  $P$ is a polynomial.
	
	If the set $\C \backslash P^{-1}(\cal{W})$ is not connected, then there is a bounded closed set $\Lambda \subset \C \backslash P^{-1}(\cal{W})$, which is surrounded by $P^{-1}(\cal{W})$. Take $\lambda\in \Lambda $, we have $P(\lambda)\notin \cal{W}$. Because $\cal{W}$ is simply connected, we can start from $P(\lambda)$ and make a smooth curve $\gamma$ that tends to infinity, such that $\gamma\cap \cal{W}=\emptyset$ and thus  $P^{-1}(\gamma)\cap P^{-1}(\cal{W})=\emptyset$.
	
	On the other hand, $\gamma$ is unbounded, so the band $\gamma_n$ of $P^{-1}(\gamma)$ containing $\lambda$ is also unbounded.
	While $P^{-1}(\cal{W})$ is bounded, so $\gamma_n\cap P^{-1}(\cal{W})\neq \emptyset$ and  $P^{-1}(\gamma)\cap P^{-1}(\cal{W})\neq \emptyset$, which is a contradiction.
\end{proof}

From now on, let's focus on that $\gamma$  is a simple closed curve.
Without losing generality, we always assume $\gamma$ is counterclockwise when the parameter increases, otherwise we reverse the parameter.

\begin{Corollary}\label{band deco2}
	Let $\gamma=\{z(t)|t\in (\alpha, \beta ]\}$ be a piecewise smooth simple closed curve, where $z(t),t\in \R$ is a periodic function and $z(\alpha)=z(\beta)$.
	Then
	
	{\rm(1)}	$P^{-1}(\gamma)$ has a unique decomposition of petals.
	
	{\rm(2)}	$P^{-1}(\gamma)$  has a unique band decomposition, such that each petal is composed of bands end to end.
\end{Corollary}
\begin{proof}
	By Corollary \ref{band deco1}, we have a band decomposition $P^{-1}(\gamma)=\bigcup^N_{n=1}{\gamma_n}$, with $\gamma_n=\{\lambda_n(t)|t\in (\alpha, \beta ]\}$. 
	Because $z(\alpha)=z(\beta)$, we have $\lambda_{n}(\alpha),\lambda_{n}(\beta )\in P^{-1}(z(\alpha)),n=1,2,\cdots,N$.
	For any band $\gamma_{n_1}$, if $\lambda_{n_1}(\alpha)=\lambda_{n_1}(\beta )$, then this band is a closed curve thus a petal. Otherwise, there is another  band $\gamma_{n_2}$, with $\lambda_{n_2}(\alpha)=\lambda_{n_1}(\beta )$. If $\lambda_{n_1}(\alpha)=\lambda_{n_2}(\beta )$, then these two bands composed a closed curve. Otherwise, continue  this process, at most after  $s$ steps, $s\leq N$, we can get a closed curve $C_1$, which is composed of bands  $\gamma_{n_1},\gamma_{n_2},\cdots,\gamma_{n_s}$ end to end. If $P^{-1}(\gamma) \backslash C_1 \neq \emptyset$, similarly, we can get  a closed curve $C_2 \subset P^{-1}(\gamma) \backslash C_1 $. Repeat the above process, we can decompose $P^{-1}(\gamma)$ into closed curves $C_1,C_2,\cdots,C_k,k\leq N$.
	
	Let $W$ be the domain bounded by $\gamma$.  To ensure that these $k$ closed curves are all petals, we only need $U_1(z(t))$ in the Corollary \ref{band deco1} to be included in $W$ for all $t\in (\alpha, \beta ]$. 
	In other words, $U_1(z(t))$ is the left neighborhood along the increasing direction of $t$, for $\gamma$ rotates counterclockwise with respect to parameter. In this case, the band decomposition method is unique, and we have (1) and (2).
\end{proof}

It is obvious that $P^{-1}(\gamma)$ has a unique bouquet decomposition, where every bouquet is composed of a petal or some flowers.  By Proposition \ref{local1} and Corollary \ref{band deco2}, we have 
\begin{Corollary}\label{flower}
	$\lambda$ is a center of a flower if and only if $\lambda\in \cal{S}\cap  P^{-1}(\gamma)$. Hence the number of flowers  in $ P^{-1}(\gamma)$ is $\#( \cal{S}\cap  P^{-1}(\gamma))$. If $\lambda$ is a center of a flower, there are $\tau(\lambda)+1$ petals in this flower, and the small neighborhood of $\lambda$ is divided into $2\tau(\lambda)+2$  regions by these $\tau(\lambda)+1$ petals,  each region has an angle   $\frac{\pi}{\tau(\lambda)+1}$.
\end{Corollary}

To get more precise information about $P^{-1}(\gamma)$, we now figure out the number of petals and bouquets as well as  how each petal is composed of bands.

\begin{Lemma}\label{global0+}
	Let $\gamma$ be a piecewise smooth simple closed curve and  $C\subset P^{-1}(\gamma)$ be a petal. Then $P(C)=\gamma$ and $P(\mathring{D})=\mathring{W}$, where $W, D$ are the domains bounded by $\gamma, C$ respectively.  	
\end{Lemma}
\begin{proof}
	By Proposition \ref{local1}, $P^{-1}(\gamma)= \partial  P^{-1}(W)$. By Corollary \ref{band deco2}, there exists a band in $C$, so $P(C)=\gamma$.
	
	For any $\lambda\in \mathring{D}$,  we have $P(\lambda)\notin \gamma$.
	If $P(\lambda)\notin \mathring{W}$, then $P(\lambda)\in \C\backslash W$.  Since $W$ is simply connected, we can start from $P(\lambda)$ and make a smooth curve $L$ that tends to infinity, such that $L\cap \gamma=\emptyset$ and thus $P^{-1}(L)\cap P^{-1}(\gamma)=\emptyset$.
	On the other hand, $L$ is unbounded, so 
	the band $L_n$ of $P^{-1}(L)$ containing $\lambda$ is also unbounded.
	While $C$ is bounded, so $L_n\cap C\neq \emptyset$ and $L\cap \gamma=P(L_n\cap C)\neq \emptyset$, which is a contradiction. Thus, $P(\mathring{D})\subset \mathring{W}$.
	
	Take $\lambda\in \mathring{D}$, then $z=P(\lambda)\in \mathring{W}$. For any $z_0\in \mathring{W}$, we can make a smooth curve $\hat{L} $ connecting points $z_0$ and $z$ in $\mathring{W}$, then $\hat{L} \cap \gamma=\emptyset$. So the band $\hat{L}_n$ of $P^{-1}(\hat{L})$ containing $\lambda$ is in $D$. Let $\lambda_0=\hat{L}_n\bigcap P^{-1}(z_0)$, then $\lambda_0\in \mathring{D}$ and $P(\lambda_0)=z_0$. Thus, $P(\mathring{D})= \mathring{W}$.
\end{proof}
\begin{Remark}\label{global0++}
	By Lemma \ref{global0+}, all petals are counterclockwise when the parameter increases.
\end{Remark}
\begin{Lemma}\label{global1}
	Let $\gamma_1,\gamma_2$ be piecewise smooth simple closed curves, and $W_1, W_2$ be the domains bounded by them respectively.  If  $W_1\subset W_2$, then $\cal{C}_P(\gamma_1)\geq\cal{C}_P(\gamma_2)$ and $\cal{B}_P(\gamma_1)\geq\cal{B}_P(\gamma_2).$
\end{Lemma}

\begin{proof}
	For any petal $C_2$ of $P^{-1}(\gamma_2)$, let $D_2$ be the domain bounded  by $C_2$. By Lemma \ref{global0+}, $P(D_2)=W_2$, so $P(D_2\cap  P^{-1}(\gamma_1))=W_2\cap  \gamma_1=\gamma_1$. That is, there is a petal $C_1$ of $P^{-1}(\gamma_1)$ in $D_2$.  Thus, $	\cal{C}_P(\gamma_1)\geq\cal{C}_P(\gamma_2)$.
	
	Similarly, 	for any bouquet $\Gamma _2$ of $P^{-1}(\gamma_2)$, let $\Omega _2$ be the domain bounded  by $\Gamma_2$, 
	there is a bouquet $\Gamma _1$ of $P^{-1}(\gamma_1)$ in $\Omega_2$. Thus, $\cal{B}_P(\gamma_1)\geq\cal{B}_P(\gamma_2).$
\end{proof}

If $E\subset P^{-1}(\gamma)$ can be decomposed into $k$ bands, we define  $\mu(E)=k$ for simplicity.

\begin{Lemma}\label{band number0}  
	Let $\gamma$ be a piecewise smooth simple closed curve and $C\subset  P^{-1}(\gamma)$ be a petal.
	%$D$ is the domain bounded  by $C$. 
	Then for any $\lambda_0\in \mathring{D}$,
	\begin{equation*}\label{bandnumber00}
		\mu(C)=\frac{1}{2\pi i}\int_{C} \frac{P'(\lambda)}{P(\lambda)-P(\lambda_0)}d\lambda,
	\end{equation*}	
	where $D$ is the domain bounded by $C$.  	
\end{Lemma}
\begin{proof}
	For any $\lambda_0\in \mathring{D}$, by Lemma \ref{global0+}, we have $P(\lambda_0)\in \mathring{W}$. By Corollar \ref{band deco2},
	$C$ is composed of $\mu(C)$ bands end to end.  By Remark \ref{global0++}, $C$ is counterclockwise, so the number of zeros of  $P(\lambda)-P(\lambda_0)$ in $ \mathring{D}$ is
	\begin{equation*}
		\frac{1}{2\pi i}\int_{C} \frac{P'(\lambda)}{P(\lambda)-P(\lambda_0)}d\lambda=	\frac{\mu(C)}{2\pi i}\int_{\gamma} \frac{1}{z-P(\lambda_0)}dz=\mu(C).    
	\end{equation*}

\end{proof}

Take any $z_0\in \gamma$. $U$ is the neighborhood of $z_0$ in Proposition \ref{local1}, which is divided into $U_1,U_2$ by curve $\gamma$, $U_1\subset W$.
Let $W_-,W_+$ be the closure of $W\backslash   U$ and $W\cup U$ respectively, $\gamma_{\pm }=\partial W_{\pm }$, then $\gamma_{\pm }$ are piecewise smooth simple closed curves. We call  $\gamma_+$($\gamma_-$) a local expansion(contraction) surgery of $\gamma$ with $U$ at $z_0$ .

\begin{Lemma}\label{local2}Let $\gamma$ be a piecewise smooth simple closed curve. Then 
	
	{\rm(1)} $	\cal{C}_P(\gamma_+)+m_P(\C,\{z_0\})
	=\cal{C}_P(\gamma)=\cal{C}_P(\gamma_-)\ \ {and }\ \ \cal{B}_P(\gamma_+)=\cal{B}_P(\gamma)=\cal{B}_P(\gamma_-)-m_P(\C,\{z_0\}).$	
	
	{\rm(2)}	Let $C\subset  P^{-1}(\gamma)$ be a petal and $D$ be the domain bounded by $C$.
	Then there is a unique petal $C_-$ of $P^{-1}(\gamma_-)$ in $D$, with $\mu(C)=\mu(C_-).$
	
	{\rm(3)}	Let $C_+\subset  P^{-1}(\gamma_+)$ be a petal, $l=1+m(\mathring{D}_+,\{z_0\})$, where
	$D_+$ is the domain bounded  by $C_+$. 
	Then there are $l$ petals $C_1,C_2,\cdots,C_l$ of $P^{-1}(\gamma)$ in  $D_+$ merged into $C_+$, with $\mu(C_+)=\sum_{j=1}^l\mu(C_j).$
\end{Lemma}
\begin{proof} {\rm(1)} 
	At any $\lambda\in P^{-1}(z_0)$, by Proposition \ref{local1}, through one-step local expansion,  $\tau(\lambda)$ petals  are reduced, and the number of bouquets  remains unchanged. So $\cal{B}(P,\gamma_+)=\cal{B}(P,\gamma)$, and 
	$$\cal{C}_P(\gamma)-\cal{C}_P(\gamma_+)=\sum_{\lambda \in   P^{-1}(z_0)} \tau(\lambda)=m_P(\C,\{z_0\}).$$
	Similarly, at any $\lambda\in P^{-1}(z_0)$, through one-step local contraction,  $\tau(\lambda)$  bouquets are added , and  the number of petals remains unchanged. So $\cal{C}_P(\gamma)=\cal{C}_P(\gamma_-)$, and 
	$\cal{B}_P(\gamma_-)-\cal{B}_P(\gamma)=m_P(\C,\{z_0\}).$ 
	
	{\rm(2)}  For $W_-\subset W$, by Lemma \ref{global1}, there exists a petal $C_-\subset  P^{-1}(\gamma_-)$ in $D$. Because $\cal{C}_P(\gamma)=\cal{C}_P(\gamma_-)$, so the existence of petal $C_-$ is unique. Take any $\lambda_0\in \mathring{D}_-$, then $P(\lambda_0)\notin U_1(z_0)$, by Lemma \ref{band number0}, we have 
	\begin{align*}
		\mu(C_-)&=\frac{1}{2\pi i}\int_{C-} \frac{P'(\lambda)}{P(\lambda)-P(\lambda_0)}d\lambda\\
		&=\frac{1}{2\pi i}\int_{C} \frac{P'(\lambda)}{P(\lambda)-P(\lambda_0)}d\lambda
		-\frac{1}{2\pi i}\int_{D\cap P^{-1}(\partial U_1(z_0))} \frac{P'(\lambda)}{P(\lambda)-P(\lambda_0)}d\lambda\\
		&=\mu(C).
	\end{align*}

	{\rm(3)} 
	At any $\lambda\in \mathring{D}_+\cap P^{-1}(z_0)$, by expansion sergery, $\tau(\lambda)$ petals  are reduced. So there are $l=1+m(\mathring{D}_+,\{z_0\})$  petals of $P^{-1}(\gamma)$ in $D_+$. Take any $\lambda_0\in \mathring{D}$, then $P(\lambda_0)\notin U_2(z_0)$, by Lemma \ref{band number0}, we have
	\begin{align*}
		\mu(C_+)&=\frac{1}{2\pi i}\int_{C_+} \frac{P'(\lambda)}{P(\lambda)-P(\lambda_0)}d\lambda\\
		&=\frac{1}{2\pi i}\int_{C_{1}} \frac{P'(\lambda)}{P(\lambda)-P(\lambda_0)}d\lambda+\cdots+\frac{1}{2\pi i}\int_{C_{l}} \frac{P'(\lambda)}{P(\lambda)-P(\lambda_0)}d\lambda\\
		&+\frac{1}{2\pi i}\int_{D_+\cap P^{-1}(\partial U_2(z_0))} \frac{P'(\lambda)}{P(\lambda)-P(\lambda_0)}d\lambda\\
		&=\sum_{j=1}^l\mu(C_j).  
	\end{align*}
\end{proof}
\begin{Remark}
	If $z_0\notin P(\cal{S})$, i.e.,  $m(\C,\{z_0\})=0$, then for any petal $C_+\subset  P^{-1}(\gamma_+)$ ,	$m(\mathring{D}_+,\{z_0\})=0$.
	Hence there is a unique petal $C$ of $P^{-1}(\gamma)$ in $D_+$, with $\mu(C_-)=\mu(C)=\mu(C_+).$
\end{Remark}

\begin{Lemma}\label{global2}
	Let $\gamma_1,\gamma_2$ be piecewise smooth simple closed curves, and $W_1, W_2$ be the domains bounded by them respectively. If $W_1\subset \mathring{W}_2$ and $P(\cal{S})\cap W_2 \backslash \mathring{W}_1=\emptyset$, then
	\begin{equation}\label{p2.3}
		\cal{C}_P(\gamma_1)=\cal{C}_P(\gamma_2)\ \ {and }\ \ \cal{B}_P(\gamma_1)=\cal{B}_P(\gamma_2).
	\end{equation}
	Moreover, let $C_2\subset  P^{-1}(\gamma_2)$ be a petal and $D_2$ be the domain bounded by $C_2$,
	then  there exists a unique petal $C_1\subset  P^{-1}(\gamma_1)$ in $D_2$, with $\mu(C_1)=\mu(C_2)$.
\end{Lemma}
\begin{proof}
	Because $P(\cal{S})\cap W_2 \backslash \mathring{W_1}=\emptyset$, 
	that is $m_P(\C,W_2)=m_P(\C,\mathring{W}_1)$, 
	so $m_P(\C,\{z\})=0$, for all $z\in W_2 \backslash \mathring{W_1}$. We can  make finite steps local expansion of $\gamma_1$ and get $\gamma_3$, such that $W_2\subset W_3$. 
	By Lemma \ref{local2}, $\cal{C}_P(\gamma_1)=\cal{C}_P(\gamma_3)$ and $\cal{B}_P(\gamma_1)=\cal{B}_P(\gamma_3)$.
	On the other hand, for $W_1\subset W_2\subset W_3$, by Lemma \ref{global1}, we have  $\cal{C}_P(\gamma_1)\geq\cal{C}_P(\gamma_2)\geq\cal{C}_P(\gamma_3)$ and $\cal{B}_P(\gamma_1)\geq\cal{B}_P(\gamma_2)\geq\cal{B}_P(\gamma_3)$. Thus, (\ref{p2.3}) holds.
	
	By Lemma \ref{global1}, there is a petal $C_1\subset  P^{-1}(\gamma_1)$ in $D_2$, because $\cal{C}_P(\gamma_1)=\cal{C}_P(\gamma_2)$, so the existence of petal $C_1$ is unique. Take any $z_0\in \mathring{W}_1 $, then $z_0\notin P(D_2 \backslash \mathring{D}_1)$, by Lemma \ref{band number0}, we have 
	\begin{align*}
		\mu(C_1)&=\frac{1}{2\pi i}\int_{C_1} \frac{P'(\lambda)}{P(\lambda)-z_0}d\lambda=\frac{1}{2\pi i}\int_{C_2} \frac{P'(\lambda)}{P(\lambda)-z_0}d\lambda=\mu(C_2).  
	\end{align*}
\end{proof}

\begin{Lemma}\label{global3}
	Let $\gamma$ be a piecewise smooth simple closed curve, and $W$ be the  domain bounded by $\gamma$. If $P(\cal{S})\cap \gamma=\emptyset$, then
	\begin{equation}\label{CBnumber0}
		\cal{C}_P(\gamma)= \cal{B}_P(\gamma)=N-m_P(\C,\mathring{W}).
	\end{equation}	
	Moreover, let $C$ be a petal in  $P^{-1}(\gamma)$, and $D$ be the  domain bounded by $C$,
	then 
	\begin{equation}\label{bandnumber0}
		\mu(C)=1+m_P( \mathring{D},\C).
	\end{equation}	
\end{Lemma}
\begin{proof}
	We prove this Lemma by induction. Let $\kappa=m_P(\C,\mathring{W}).$
	
	When $\kappa=0$, for any $z_0\in \mathring{W}$ we have $z_0\notin P(\cal{S})$, so $P(\lambda)=z_0$ has $N$ different solutions $\lambda_n,n=1,2,\cdots,N$ in $\C$. Take a sufficiently small closed neighborhood $U_n$ for each point $\lambda_n$, such that $P(U_n)\subset \mathring{W}$ and the $N$ neighborhoods do not intersect each other. Let $W_0=P(U_1)\cap P(U_2)\cap \cdots \cap P(U_N)$, with boundary curve $\gamma_0$. Then $W_0\subset \mathring{W}$ is a closed domain, and $\cal{C}_P(\gamma_0)= \cal{B}_P(\gamma_0)=N$. By Lemma \ref{global2}, we have (\ref{CBnumber0}). 
	
	Suppose that (\ref{CBnumber0}) is true for  $\kappa\leq m$. We now prove it is true for  $\kappa= m+1$. 
	Let  $P(\cal{S})\cap \mathring{W}=\{z_1,z_2,\cdots,z_s\}$, with $s\leq m+1$. Then 
	$$\sum_{j=1}^{s}m_P(\C,\{z_j\})=m_P(\C,\mathring{W})=m+1.$$
	Take a piecewise smooth simple closed curve $\gamma_*\subset \mathring{W}$ such that $z_s\in \gamma_*$ and $\{z_1,z_2,\cdots,z_{s-1}\}\subset \mathring{W}_*$, where ${W}_*$ is the domain bounded by $\gamma_*$.
	Let $\gamma_+$($\gamma_-$) be one step local expansion(contraction) of $\gamma_*$ at $z_s$ with a small closed neighborhood $U(z_s)$, such that $\gamma_+\subset \mathring{W}.$  Then
	\begin{equation}\label{global3-region}
		\{z_1,z_2,\cdots,z_{s-1}\}\subset \mathring{W}_-\subset W_*\subset W_+\subset \mathring{W},
	\end{equation}	
	where $W_+, W_-$ are the domains bounded by $\gamma_+,\gamma_-$ respectively.
	
	Because  $P(\cal{S})\cap \gamma=\emptyset$, hence $P(\cal{S})\cap W \backslash \mathring{W}_+ =\emptyset$, by Lemma \ref{global2}, we have 	$\cal{C}_P(\gamma)=\cal{C}_P(\gamma_+)$ and $\cal{B}_P(\gamma)=\cal{B}_P(\gamma_+)$.
	Because $P(\cal{S})\cap \gamma_-=\emptyset$, according to induction, 
	\begin{equation*}
		\cal{C}_P(\gamma_-)= \cal{B}_P(\gamma_-)=N-\sum_{j=1}^{s-1}m_P(\C,\{z_j\}).
	\end{equation*}	
	By Lemma \ref{local2}, we have 
	\begin{align*}
		\cal{C}_P(\gamma)&=\cal{C}_P(\gamma_+)
		=\cal{C}_P(\gamma_-)-m_P(\C,\{z_s\})\\
		&=N-\sum_{j=1}^{s}m_P(\C,\{z_j\})
		=N-(m+1),
	\end{align*}
	and 
	\begin{align*}
		\cal{B}_P(\gamma)&=\cal{B}_P(\gamma_+)
		=\cal{B}_P(\gamma_-)-m_P(\C,\{z_s\})\\
		&=N-\sum_{j=1}^{s}m_P(\C,\{z_j\})
		=N-(m+1).
	\end{align*}

	Now we prove (\ref{bandnumber0}) by induction.  When $\kappa=0$, there are exactly $N$ disjiont petals in $P^{-1}(\gamma_0)$, thus (\ref{bandnumber0}) holds.
	
	Suppose that (\ref{bandnumber0}) is true for $\kappa\leq m$. We now prove it is true for $\kappa= m+1$. \\
	By Lemma \ref{global2},  there exists a unique petal $C_+\subset  P^{-1}(\gamma_+)$ in $D$, with $\mu(C_+)=\mu(C)$.  Let $D_+$ be the domain bounded by $C_+$, by Lemma \ref{local2}(2)(3), $C_+$ splited into $l=1+m(\mathring{D}_+,\{z_s\})$ petals $C_{1-},C_{2-},\cdots,C_{l-}$  of $P^{-1}(\gamma_-)$, and $\mu(C_+)=\sum_{j=1}^l\mu(C_{j-}).$ 
	Let ${D}_{j-}$ be the domain bounded by ${C}_{j-},j=1,2,\cdots,l$. By (\ref{global3-region}), $m_P(\mathring{D}_+,\C)=m_P( \mathring{D},\C)$ and
	$$\sum _{j=1}^lm_P(\mathring{D}_{j-},\C)
	=\sum _{j=1}^l\sum _{k=1}^{s-1}m_P(\mathring{D}_{j-},\{z_k\})=\sum _{k=1}^{s-1}m_P(\mathring{D}_+,\{z_k\}).$$
	According to induction, $\mu(C_{j-})=1+m_P(\mathring{D}_{j-},\C),j=1,2,\cdots,l$.
	Thus,
	\begin{align*}
		\mu(C)&=\mu(C_+)=\sum_{j=1}^l\mu(C_{j-})\\
		&=\sum _{j=1}^l(1+m_P(\mathring{D}_{j-},\C)))=l+\sum _{k=1}^{s-1}m_P(\mathring{D}_+,\{z_k\})\\
		&=1+\sum _{k=1}^{s}m_P(\mathring{D}_+,\{z_k\})=1+m_P( \mathring{D}_+,\C)=1+m_P( \mathring{D},\C).    
	\end{align*}
\end{proof}

\begin{Proposition}\label{global4}Let $\gamma$ be a piecewise smooth simple closed curve, $C$ be a petal in  $P^{-1}(\gamma)$, $\Gamma$ be a bouquet in $P^{-1}(\gamma)$, and $W, D,\Omega$ be the domains bounded by them respectively. 
	Then
	
	{\rm(1)} $\cal{C}_P(\gamma)=N- m_P(\C,\mathring{W} )$ and $\cal{B}_P(\gamma)=N- m_P(\C, W)$.
	
	{\rm(2)}	The number of bands contained in $C$ is $\mu(C)=1+m_P( \mathring{D},\C).$	
	
	{\rm(3)}	The number of flowers in $\Gamma$ is $\#(\cal{S}\cap \Gamma)$.
	
	{\rm(4)}	The number of petals  contained in $\Gamma$ is 
	$1+m_P(\Gamma,\C).$	
	
	{\rm(5)}	 The number of bands  contained in $\Gamma$ is 
	$\mu(\Gamma)=1+m_P(\Omega,\C).$ 
\end{Proposition}
\begin{proof}
	{\rm(1)} 	Suppose $P(\cal{S})\cap \gamma=\{z_1,z_2,\cdots,z_s\}$ and $P(\cal{S})\cap \mathring{W}=\{z_{s+1},z_{s+2},\cdots, z_{s+t}\}$. 
	For $j=1,2,\cdots,s$, let $z_j$ be the center points with the small neighborhood $U(z_j)$, we can  make $s$ steps local  expansion(contraction) of $\gamma$ and get $\gamma_1$($\gamma_2$),  then $P(\cal{S})\cap \gamma_1=\emptyset$ and  $\{z_1,z_2,\cdots,z_{s+t}\}\subset \mathring{W}_1$, $P(\cal{S})\cap \gamma_2=\emptyset$
	and $\{z_{s+1},\cdots,z_{s+t}\}\subset \mathring{W}_2\subset W\subset W_1,$ where $W_1, W_2$ are the domains bounded by $\gamma_1,\gamma_2$ respectively.
	By Lemma \ref{global3}, we have $$\cal{C}_P(\gamma_1)= \cal{B}_P(\gamma_1)=N-m_P(\C, \mathring{W}_1)=N-m_P(\C, W).$$ 
	Moreover, by Lemma \ref{local2}, we have $\cal{B}_P(\gamma)=\cal{B}_P(\gamma_1)=
	N- m_P(\C, W)$ and 
	\begin{align*}
		\cal{C}_P(\gamma)&=\cal{C}_P(\gamma_1)+\sum _{j=1}^{s}m_P(\C,\{z_j\})\\
		&=N-m_P(\C, W)+m_P(\C,\gamma)=N-m_P(\C, \mathring{W}).
	\end{align*}
	
	{\rm(2)} 	By Lemma \ref{local2} and Lemma \ref{global3}, there is a unique petal $C_{2}$ of $P^{-1}(\gamma_2)$ in $D$, and 
	$$\mu(C)=\mu(C_{2})=1+m_P( \mathring{D}_{2},\C)=1+m_P( \mathring{D},\C).$$

	{\rm(3)(4)} 	By Corollary \ref{flower}, $\lambda$ is a center of a flower  if and only if $\lambda\in \cal{S}\cap  P^{-1}(\gamma)$. Hence the number of flowers in $\Gamma$ is $\#(\cal{S}\cap \Gamma)$, and the number of petals  contained in $\Gamma$ is 
	$1+m_P(\Gamma,\C).$ 
	
	{\rm(5)} 	By Corollary \ref{band deco2}, there is a unique petal decomposition, hence $\Gamma=C_1\cup C_2\cup \cdots \cup C_l  $, where $l=1+m_P(\Gamma,\C).$
	% For any $j=1,2,\cdots,l$, by Lemma \ref{local2},   there is a unique petal $C_{j_2}$ of $P^{-1}(\gamma_2)$ in $D_j$, and $\mu(C_j)=\mu(C_{j_2})$. 
	By (2), $\mu(C_{j})=1+m_P( \mathring{D}_{j},\C)$.
	Thus,
	\begin{align*}\mu(\Gamma)&=\sum _{j=1}^l\mu(C_j)=\sum _{j=1}^l(1+m_P( \mathring{D}_{j},\C))\\
		&=l+m_P( \mathring{\Omega},\C)=1+m_P(\Omega,\C).  
	\end{align*}
\end{proof}

\subsection{Proof of Theorem \ref{main1}.}
For $|a|\ne |c|$, $\mathcal{E} =\{ae^{-i\theta }+ce^{i\theta }\mid \theta \in(-\pi,\pi] \}$ is an ellipse, which is analytic and counterclockwise with respect to parameter $\theta$.  Let $W$ be the domain bounded by $\mathcal{E}$.

\rm{(1)} By Corollary \ref{global0}, we have (1).

\rm{(2)} For $\cal{C}_P(\gamma)=N- m_P(\C,\mathring{W} )=1+ m_P(\C,\C \backslash \mathring{W} )$,
by Corollary \ref{band deco2}, Proposition \ref{global4}, we have (2).

\rm{(3)} By Corollary \ref{band deco2}, $\sigma(J)$ has a unique band decomposition, such that each petal is composed of bands end to end.  By Proposition \ref{local1} and  Remark \ref{analytic}, each band $\gamma_n=\{\lambda_n(\theta),\theta\in(-\pi,\pi] \}$ is  piecewise analytic curve and the no smooth point points are in $\cal{S}$.
By Lemma \ref{global0+}, $P(\mathring{D})=\mathring{W}$. By Lemma \ref{band number0} and  Proposition \ref{global4},  $\mu(C)=\frac{1}{2\pi i}\int_{C} \frac{P'(\lambda)}{P(\lambda)-P(\lambda_0)}d\lambda=1+m_P(\mathring{D},\C).$

\rm{(4)} By Corollary \ref{flower}, we  have (4).

\rm{(5)} It is obvious that $\sigma(J)$ has a unique bouquet decomposition. By (2) and (4), every bouquet is composed of a petal or some flowers. Because $N- m_P(\C,\mathring{W} )=1+ m_P(\C,\C\backslash W) $, by Proposition \ref{global4}, we have (5).        \hfill    $\square $

\subsection{Proof of Theorem \ref{main2}.}
For $|a|= |c|\ne 0$, $\cal{E}=e^{i\varphi}  \cdot [-2R,2R]$ is a  straight line segment, which can be regarded as a degenerate   ellipse. 	

\rm{(1)} By Corollary \ref{global0}, we have (1). 

\rm{(2)} By Corollary \ref{band deco1}, Remark \ref{not unique}, Proposition \ref{local1}, Remark \ref{analytic}, we have (2). 

\rm{(3)} By Proposition \ref{local1}, we have (3). 

\rm{(4)} Let $W$ be the domain bounded by an ellipse, such that $\mathcal{E}\subset W$ and $W\cap P(\cal{S})=\mathcal{E}\cap P(\cal{S})$. Then $P^{-1}(W) $ has the same global structure. Let $W$ approach $\mathcal{E}$, then by Theorem \ref{main1} we have (4). 
\hfill    $\square $

\section{Real spectrum and spectrum on straight line}
\subsection{ Real spectrum}  
In this section, as an application of the Theorem \ref{main2}, we first consider the case when the spectrum belongs to $\R$. Because the case $a=c=0$ is trivial, we still assume that $a$ and $c$ are not both $0$. 

If $J$ is a $N$-periodic self-adjoint operator, that is, $a_n=\bar{c}_n$ and $b_n\in \R$ for $n=1,2,\cdots,N$. Then $a=\bar{c}$, and  $\mathcal{E}=[-2R,2R]$, with $R=|a|=|c|$. On the other hand, through induction, we can get that the discriminant $P$ is a monic real polynomial of degree $N$. By the classical Floquet Theory, for $R\neq 0$, we have 
\begin{equation}\label{rp0}
	\sigma(J)=P^{-1}([-2R,2R])=\bigcup^N_{n=1}{[\alpha_n, \beta_n ]}\subset \R, 
\end{equation}
with 
\begin{equation}\label{rp1}
	\alpha_N< \beta_N\leq \alpha_{N-1}< \cdots<\beta_2\leq \alpha_1< \beta_1,
\end{equation}
where $[\alpha_{n},\beta_{n},], n=1,2,\cdots,N$ are $N$ bands of $\sigma(J)$.\footnote{If $R=0$, these $N$ closed intervals degenerate into $N$ spectral points, counting the multiplicity.} 
For $n=1,2,\cdots,N-1$, if $\beta_{n+1}<\alpha_{n}$, $(\beta_{n+1},\alpha_{n})$ is an open gap of $\sigma(J)$.
By the Differential Mean Value Theorem, the stationary points  $\cal{S}=\{\lambda _1,\lambda _2,...,\lambda _{N-1}\}\subset \R$, and satisfies 
\begin{equation}\label{rp2}
	\alpha_N<  \beta_N\leq \lambda_{N-1} \leq  \alpha_{N-1}<  \cdots<\beta_2 \leq \lambda_1 \leq \alpha_1< \beta_1.
\end{equation}

If $J$ is a non-self-adjoint operator, we have 
\begin{Corollary}\label{real spectrum0}
	If $a$ and $c$ are not both $0$, then $ \sigma(J)\subset \R$ if and only if there is a unique band decomposition (\ref{rp0}) with (\ref{rp1}).
\end{Corollary}
\begin{proof}
	We only need to prove the necessity. Because $ \sigma(J)\subset \R$, there is no petal in $\sigma(J)$, and  for $a$ and $c$ are not both $0$, we have $|a|= |c|\neq 0$. By Theorem \ref{main2}, there is a band decomposition
	$ \sigma(J)=\bigcup^N_{n=1}{\gamma_n }.$
	Each band $\gamma_n $ is a piecewise analytic curve. By Corollary \ref{bands intersect+}, at most one intersection between any two bands. 
	For $ \sigma(J)\subset \R$, each band $\gamma_n$ is a closed interval, and there is at most one intersection between adjacent closed intervals. Thus, the conclusion is also true.
\end{proof}

Corollary \ref{real spectrum0} gives the shape of the real spectrum. To further study real spectrum, we first give a lemma.
\begin{Lemma}\label{rp}
	Let $f(z)$ be a polynomial. If there exists $[\alpha, \beta ]\subset \R$ such that $f([\alpha, \beta ])\subset \R$,  then $f(z)$ is a real coefficient polynomial.
\end{Lemma}
\begin{proof}
	$f(z)$ is a polynomial, it can be expressed as $f(z)=f_1(z)+f_2(z)i$, where $f_1(z)$ and $f_2(z)$ are all real coefficient polynomials.  For every $t\in [\alpha, \beta ]$, $f(t)=f_1(t)+f_2(t)i\in \R$, that is, $f_2(t)=0$ for every $t\in [\alpha, \beta ]$. Thus, $f_2(t)=0$ for all $t \in \R$, and $f(z)=f_1(z)$ is a real coefficient polynomial.
\end{proof}

\noindent{\bf Proof of Theorem \ref{real spectrum}. }(Sufficiency). Because $a$ and $c$ are not both $0$, and $ \sigma(J)\subset \R$,  by Corollary \ref{real spectrum0}, we have  a band decomposition (\ref{rp0}) with (\ref{rp1}).

For $\mathcal{E}=e^{i\varphi}  \cdot[-2R,2R] $,	let 
$P(\lambda)=e^{i\varphi} f(\lambda),$
then $f([\alpha_n, \beta_n ])= [-2R,2R]$ is one-to-one for all $n=1,2,\cdots,N$.
By Lemma \ref{rp}, $f$ is a real coefficient polynomial, 
so $f'(\lambda)=0$ has $N-1$ real single roots $\{\lambda_1,\lambda_2,\cdots,\lambda_{N-1}\}$, satisfying  (\ref{rp2}).

For $P'(\lambda)=e^{i\varphi} f'(\lambda)$,   $P'(\lambda)=0$ has the same  $N-1$ real single roots,  satisfy  $\lambda_1>\lambda_2>\cdots>\lambda_{N-1}$. We have condition (2).

Because $P(\lambda)$  is a monic  polynomial of degree $N$, we have 
\begin{equation}\label{dp}
	P'(\lambda)=N(\lambda-\lambda_1)(\lambda-\lambda_2)\cdots (\lambda-\lambda_{N-1}).
\end{equation}
So 
$P(\lambda)=Q(\lambda)+C,$  where $Q(\lambda)$ is a real coefficient polynomial. To ensure that equation   $P(\lambda)=e^{i\varphi} f(\lambda)=Q(\lambda)+C$   is correct, we need $C\in \R$ and $e^{i\varphi}=\pm 1$. Thus, $P(\lambda)$ is a real coefficient polynomial. 

For  $e^{i\varphi}=\pm 1$, we have $\frac{\varphi_a+\varphi_c}{2}=\varphi=n\pi,n\in \Z$. So $\varphi_a=-\varphi_c+2n\pi$ and $a=\bar{c}$. We have condition (1).

For all $n=1,2,\cdots,N-1$, $P(\lambda_n)\in \R$. If $|P(\lambda_n)|< 2R$, for $\cal{E}=[-2R,2R], $ by Proposition \ref{local1}, there are complex spectrum near $\lambda_n$, so $|P(\lambda_n)|\geq 2R$.
By (\ref{dp}), if $\lambda_n$ is a endpoint of a band we have $(-1)^nP(\lambda_n) = 2R$, else if $\lambda_n \notin \sigma(J)$ we have  $(-1)^nP(\lambda_n) >2R$. Thus,
$(-1)^nP(\lambda_n) \geq 2R$. We have condition (3).

(Necessity).  By condition (2), we have (\ref{dp}), so $P(\lambda)=Q(\lambda)+C,$  where $Q(\lambda)$ is a real coefficient polynomial.  For $P(\lambda_1) \in \R$, we have $C \in \R$, so $P(\lambda)$ is a real coefficient polynomial. 

Because $a=\bar{c}$,  so $\cal{E}=[-2R,2R]\in \R$ and $\sigma(J)=P^{-1}([-2R,2R])$. For $P(\lambda)$ being a  monic  real coefficient polynomial, there exist $\lambda_0,\lambda_N$ such that  $\lambda_0>\lambda_1>\lambda_{N-1}>\lambda_N$, with 
$P(\lambda_0)>2R$ and $(-1)^NP(\lambda_N)>2R$. Then, for  all $ n= 0,1,2,\cdots,N$, we have $(-1)^nP(\lambda_n) \geq 2R$. By Intermediate Value Theorem, for  any $ n=1,2,\cdots,N$, there is a closed interval $\gamma_n\subset [\lambda_{n},\lambda_{n-1}]$, which is a band of $\sigma(J)$. Now we have $N$ bands of $\sigma(J)$, by Theorem \ref{main2}, $\sigma(J)=\bigcup^{N}_{n=1} \gamma_n\subset \R$.         \hfill    $\square $   

\begin{Corollary}
	If $ \sigma(J)\subset \R$, then $P(\lambda)$ is a real coefficient polynomial.\footnote{If $a=c=0$, for $ \sigma(J)=P^{-1}(0)\subset \R$, this conclusion is obviously correct.}
\end{Corollary}

In Theorem \ref{real spectrum}, if $(-1)^nP(\lambda_n) >2R$ for all $n=1,2,\cdots,N-1$, then $\lambda_n \in (\beta_{n+1},\alpha_{n})\subset \rho(J)$, that is, all gaps are open. 
Conversely, if $(-1)^nP(\lambda_n) = 2R$ for all $n=1,2,\cdots,N-1$, then $\lambda_n\in \sigma(J)$ and $\beta_{n+1}=\lambda_{n} = \alpha_{n}$. Thus we have
\begin{Corollary}\label{interval spectrum}
	$ \sigma(J)\subset \R$ is a closed interval, if and only if the following three conditions hold,
	
	{\rm(1)} $a=\bar{c}$.	
	
	{\rm(2)}   $\cal{S}=\{\lambda_1,\lambda_2,\cdots,\lambda_{N-1}\}\subset \R$, with $\lambda_1>\lambda_2>\cdots>\lambda_{N-1}$.
	
	{\rm(3)}   $(-1)^nP(\lambda_n) = 2R, n=1,2,\cdots,N-1$.	
	
\end{Corollary}

It is a special case that the spectrum is a real interval, and the condition (2)(3) in Corollary \ref{interval spectrum} can be weakened as in Theorem \ref{interval spectrum2}.

\noindent 
\textbf{Proof of Theorem \ref{interval spectrum2}.}
The sufficiency is obvious, thus we only need to prove the necessity.

By condition (1), $\cal{E}=[-2R,2R]\in \R$, with $R=|a|$.
By condition (2) and Corollary \ref{N bands end to end}, $ \sigma(J)$ is composed of $N$ bands $\gamma_1,\gamma_2,\cdots,\gamma_N$ end to end. 
By condition (3), $\alpha$ and $\beta$ are endpoints of some bands,
so $ \sigma(J)$ is a curve connecting $\alpha$ and $\beta$. Let $\cal{S}=\{\lambda_1,\lambda_2,\cdots,\lambda_{N-1}\}$, we have (\ref{dp}). Let $\lambda_0=\beta,\lambda_N=\alpha$ and for all $n=1,2,\cdots,N$,
$\lambda_n$ and $\lambda_{n-1}$ are the two endpoints of the band $\gamma_n$, then the sign of $P(\lambda_n)$ is crossed. Let's assume $(-1)^nP(\lambda_n) = 2R$, by Corollary \ref{interval spectrum} we just need to prove that $\lambda_1,\lambda_2,\cdots,\lambda_{N-1} \in\R$.

Because
$\lambda_0,\lambda_N$ are single roots of $P(\lambda) - (-1)^n2R=0$, and $\lambda_1,\lambda_2,\cdots,\lambda_{N-1}$ are double roots of $P(\lambda) - (-1)^n2R=0$.  We have found $N$ roots (counting the multiplicity) for $P(\lambda)-2R=0$. Since $P(\lambda)$ is a monic  polynomial of degree $N$, it can be written as 
\begin{equation}\label{P-2R}
	P(\lambda)-2R=(\lambda-\lambda_0)(\lambda-\lambda_2)^2(\lambda-\lambda_4)^2\cdots.
\end{equation}
Similarly, we have
\begin{equation}\label{P+2R}
	P(\lambda)+2R=(\lambda-\lambda_1)^2(\lambda-\lambda_3)^2\cdots. 
\end{equation}
Multiply (\ref{P-2R}) and (\ref{P+2R}), we have
\begin{equation}\label{P2-4R2}
	P^2(\lambda) - 4R^2=(\lambda-\lambda_0)(\lambda-\lambda_1)^2(\lambda-\lambda_2)^2\cdots (\lambda-\lambda_{N-1})^2(\lambda-\lambda_N).
\end{equation}
Combining (\ref{dp}) and (\ref{P2-4R2}), we get the differential equation 
\begin{equation}\label{wffc1}
	\frac{P'(\lambda)}{\sqrt{4R^2- P^2(\lambda)} }=\frac{  N}{\sqrt{(\lambda_0-\lambda)( \lambda-\lambda_N)}}.
\end{equation}

Now, let's solve the equation (\ref{wffc1}).
Let $\widehat{\lambda}= \frac{4}{\beta-\alpha}\cdot (\lambda -\frac{\alpha+\beta}{2})$, then $\widehat{\lambda}_{n} =\frac{4}{\beta-\alpha}\cdot (\lambda_n -\frac{\alpha+\beta}{2})$ for $n=0,1,2,\cdots,N$, and  
equation (\ref{wffc1}) reduced to 
$$\frac{dP}{\sqrt{4R^2- P^2} }=\frac{  Nd\widehat{\lambda}}{\sqrt{4-\widehat{\lambda}^2}}.$$
For any $n=1,2,\cdots,N$, integral from $\widehat{\lambda}_{n}$ to $\widehat{\lambda}_{n-1}$ on both sides, we get
\begin{align*}
	N(-\arccos(\frac{\widehat{\lambda}_{n-1}}{2})+\arccos(\frac{\widehat{\lambda}_{n}}{2}))
	&= -\arccos(\frac{P(\lambda_{n-1})}{2R}) +\arccos(\frac{P(\lambda_{n})}{2R})\\
	&=-\arccos((-1)^{n-1}) +\arccos((-1)^{n})\\
	&=\pm\pi.
\end{align*}
So we have 
$|-\arccos(\frac{\widehat{\lambda}_{n-1}}{2})+\arccos(\frac{\widehat{\lambda}_{n}}{2})|=\frac{\pi}{N}$. For 
$\widehat{\lambda}_{0}=2$ and $\widehat{\lambda}_{N}=-2$, 
$|-\arccos(\frac{\widehat{\lambda}_{N}}{2})+\arccos(\frac{\widehat{\lambda}_{0}}{2})|=\pi$. By the trigonometric inequality 
\begin{align*}
	\pi&=|-\arccos(\frac{\widehat{\lambda}_{N}}{2})+\arccos(\frac{\widehat{\lambda}_{0}}{2})|\\
	&\leq  \sum^{N}_{n=1}|-\arccos(\frac{\widehat{\lambda}_{n-1}}{2})+\arccos(\frac{\widehat{\lambda}_{n}}{2})|  =N\cdot \frac{\pi}{N}=\pi.
\end{align*}
To ensure that the equation holds, we need $\arccos(\frac{\widehat{\lambda}_{n}}{2})=\frac{n\pi}{N},n=0,1,2,\cdots,N$. Thus
$\widehat{\lambda}_{n}=2\cos(\frac{n\pi}{N})$, and 
$\lambda_{n}=\frac{\alpha+\beta}{2}+\frac{\beta-\alpha}{2}\cdot \cos(\frac{n\pi}{N})\in \R$.    \hfill    $\square $    

From a geometric point of view, by Corollary \ref{N bands end to end}, the three conditions of the Theorem \ref{interval spectrum2} mean that if the spectrum $ \sigma (J) $ is a curve connecting two different real numbers $ \alpha $ and $ \beta $, then $ \sigma(J)$ is the line segment connecting them. In \cite{VGP}, V. G. Papanicolaou proved a similar conclusion in different ways under special cases ($a_n=c_n,n=1,2,\cdots,N$).

\noindent{\bf Proof of Theorem \ref{interval spectrum3}. } 
(Sufficiency). For $ \sigma(J)=[\alpha,\beta],a=\bar{c}$ and $\cal{E}=[-2R,2R]\in \R$, with $R=|a|$.
In the proof of Theorem \ref{interval spectrum2}, we have got $\widehat{\lambda}_{n}=2\cos(\frac{n\pi}{N})$, and 
$\lambda_{n}=\frac{\alpha+\beta}{2}+\frac{\beta-\alpha}{2}\cdot \cos(\frac{n\pi}{N})$, for $n=0,1,2,...,N$. Substitute  the transformation $\widehat{\lambda}= \frac{4}{\beta-\alpha}\cdot (\lambda -\frac{\alpha+\beta}{2})$ into (\ref{P-2R}) and (\ref{P+2R}), we have 
\begin{align}\label{P-2R1}
	\nonumber
	P(\lambda)-2R&=(\lambda-\lambda_0)(\lambda-\lambda_2)^2(\lambda-\lambda_4)^2\cdots\\
	\nonumber
	&=(\frac{\beta-\alpha}{4})^N\cdot (\widehat{\lambda}-2)(\widehat{\lambda}-2\cos(\frac{2\pi}{N}))^2(\widehat{\lambda}-2\cos(\frac{4\pi}{N}))^2\cdots.
\end{align}
By Example \ref{unp},  ${P}_N(\lambda)-2=(\lambda-2)(\lambda-2\cos(\frac{2\pi}{N}))^2(\lambda-2\cos(\frac{4\pi}{N}))^2\cdots.$ So, we have 
\begin{equation}\label{P-2R1}
	P(\lambda)-2R=(\frac{\beta-\alpha}{4})^N\cdot(P_N(\widehat{\lambda})-2).
\end{equation}
Similarly, we have
\begin{equation}\label{P+2R1}
	P(\lambda)+2R=(\frac{\beta-\alpha}{4})^N\cdot(P_N(\widehat{\lambda})+2).
\end{equation}
Combine (\ref{P+2R1}) with  (\ref{P-2R1}), we have $R=(\frac{\beta-\alpha}{4})^N$ and 
\begin{equation*}\label{P1}
	P(\lambda)=(\frac{\beta-\alpha}{4})^N\cdot P_N( \frac{4}{\beta-\alpha}\cdot (\lambda -\frac{\alpha+\beta}{2})).
\end{equation*} 

(Necessity).  
For $\cal{E}=[-2R,2R]$, thus $a=\bar{c}$.
Because  $P_N$ has stationary points $ \{2\cos(\frac{n\pi}{N})|n=1,2,\cdots,N-1\},$
so $$ \cal{S}=\{\frac{\alpha+\beta}{2}+\frac{\beta-\alpha}{2}\cos(\frac{n\pi}{N})|n=1,2,\cdots,N-1\},$$ 
$\alpha, \beta\notin \cal{S}$ and $\#\cal{S}=N-1$.
Through direct calculation, there are $P(\cal{S} \cup \{\alpha,\beta\})\subset \{\pm 2R\}$. By Theorem \ref{interval spectrum2}, $ \sigma(J)=[\alpha,\beta]$.  \hfill    $\square $

\begin{Corollary}\label{interval spectrum1} 
	$\sigma(J)=[-2,2]\Leftrightarrow  \cal{E}=[-2,2]$ and $P=P_N$.
\end{Corollary}
By Corollary \ref{interval spectrum1}, if we want to find operator $J$ such that its spectrum $\sigma(J)$ is interval $[-2,2]$, we can  take $a_n,c_n,n=1,2,\cdots,N$ arbitrarily first, such that $a=\bar{c}$ and $R=|a|=|c|=1$,  and then look for $b_n,n=1,2,\cdots,N$, such that $P={P}_N$.
\begin{Example}
	In this example, we will look for a non-self-adjoint $3$-periodic  Jacobi operator $J$, such that $\sigma(J)=[-2,2]$.
	
	Let $a_1=\frac{1}{2\sqrt{2}}e^\frac{\pi i}{3},a_2=2i,a_3=\sqrt{2}e^\frac{\pi i}{4},c_1=\frac{i}{\sqrt{2}},c_2=e^\frac{\pi i}{12},c_3=\sqrt{2}e^\frac{\pi i}{3}$. Then $a=\bar{c}$ and $R=1$. 
	Direct calculation shows that $P(\lambda)=(\lambda-b_1)(\lambda-b_2)(\lambda-b_3)-a_1c_1(\lambda-b_3)-a_2c_2(\lambda-b_1)-a_3c_3(\lambda-b_2)$. 
	To ensure $P(\lambda)={P}_N(\lambda)=\lambda^3-3\lambda$, by comparing the coefficients of power $\lambda$, 
	we can get several groups of solutions, for example, we can take  $b_1=\sqrt{3-a_1c_1-a_2c_2-a_3c_3},b_2=-b_1,b_3=0$. Then, $\sigma(J)=[-2,2]$. 
\end{Example}

\subsection{ Spectrum on straight line} 
Now, let's consider the case of $ \sigma(J)$ belongs to a straight line in the complex plane $\C$.
\begin{Proposition}\label{line spectrum}
	If $a$ and $c$ are not both $0$, $ \sigma(J)\subset L \subset \C$, where	$L=\{C_0+e^{i\phi}\cdot t| t\in \R\}$ is a straight line in $\C$.
	Then $|a|=|c|\neq 0 $ and 
	\begin{equation}\label{angle}
		\phi \in \{\frac{\varphi+n\pi}{N}|n\in \Z\},
	\end{equation}
	where $\varphi=\frac{\varphi_a+\varphi_c}{2}$.
\end{Proposition}
\begin{proof}
	Let $\widehat{J} =e^{-i\phi}\cdot (J-C_0Id),$ then $\widehat{a}_n=e^{-i\phi} a_n,\widehat{c}_n=e^{-i\phi} c_n,\widehat{b}_n=e^{-i\phi} (b_n-C_0)$ and we have 
	$$\lambda \in \sigma(J)\subset L \Leftrightarrow e^{-i\phi}\cdot (\lambda -C_0)\in \sigma(\widehat{J} )\subset \R .$$
	
	For $a$ and $c$ are not both $0$, by Theorem \ref{real spectrum}, we have 
	$\widehat{a}= \bar{\widehat{c}} $, where $\widehat{a}=\widehat{a}_1\widehat{a}_2\cdots \widehat{a}_N=e^{-iN\phi} a=Re^{i(\varphi_a-N\phi)}$ and $\widehat{c}=Re^{i(\varphi_c-N\phi)}$. Thus, there is $n\in \Z$, such that 
	$\varphi_a-N\phi+2n\pi=-(\varphi_c-N\phi)$, that is, $\phi=\frac{\varphi_a+\varphi_c+2n\pi}{2N}=\frac{\varphi+n\pi}{N}$.
\end{proof}

By Proposition \ref{line spectrum}, if $a$ and $c$ are given, or $\varphi=\frac{\varphi_a+\varphi_c}{2}$ is given first, then the direction of line $L$ must satisfy (\ref{angle}). For example, let $N=2,\varphi=0$, then the direction of $L$  can only be vertical and horizontal.

By Proposition \ref{line spectrum}, it is easy to construct an example where the spectrum $\sigma(J)$ belongs to the line $L(t)=C_0+e^{i\phi}\cdot t, t\in \R$. We can take any $\widehat{J}$, such that $\sigma(\widehat{J} )\subset \R $. Let $J=e^{i\phi}\cdot \widehat{J}+C_0I $, then $ \sigma(J)\subset L$.

\begin{Proposition}\label{line segment spectrum}
	Let $\alpha,\beta\in \C$ be two distinct numbers, and $\beta-\alpha=r_0e^{i\phi}$ be the corresponding polar coordinate. Then the following three assertions are equivalent.
	
	{\rm(1)}	$ \sigma(J)$ is a curve connecting $\alpha$ and $\beta$.
	
	{\rm(2)}	$ \sigma(J)$ is the line segment connecting $\alpha$ and $\beta$. 	
	
	{\rm(3)}		$\cal{E}=e^{i\varphi}  \cdot  [-2R,2R]$, and $P(\lambda)=(\frac{\beta-\alpha}{4})^N\cdot {P}_N( \frac{4}{\beta-\alpha}\cdot (\lambda -\frac{\alpha+\beta}{2})),$ where  $R=(\frac{r_0}{4})^N,\varphi=N\phi-n\pi,n\in\Z$.
	
\end{Proposition}

\begin{proof}
	$(1) \Rightarrow (2)$: Let 
	\begin{equation}\label{trans}
		\widehat{J} =\frac{4}{\beta-\alpha}\cdot (J-\frac{\alpha+\beta}{2}Id),
	\end{equation}
	we have $\widehat{\alpha}=-2,\widehat{\beta}=2$ and
	$$\lambda \in \sigma(J) \Leftrightarrow \widehat{\lambda}= \frac{4}{\beta-\alpha}\cdot (\lambda -\frac{\alpha+\beta}{2})\in \sigma(\widehat{J} ).$$
	For $ \sigma(\widehat{J})$ being a curve connecting $\widehat{\alpha}$ and $\widehat{\beta}$, by Theorem \ref{interval spectrum2}, $ \sigma(\widehat{J})=[-2,2]$ is a line segment, so $\sigma(J)$ is a line segment connecting $\alpha$ and $\beta$. 
	
	$(2) \Rightarrow (3)$: $ \sigma(J)$ is the line segment joining $\alpha$ and $\beta$, so $|a|=|c|\neq 0$ and $\cal{E}=e^{i\varphi}  \cdot  [-2R,2R]$.  By the transformation (\ref{trans}), $\sigma(\widehat{J})=[-2,2]$. 
	By Corollary \ref{interval spectrum1}, $\widehat{P}={P}_N$ and $\widehat{\cal{E}}=[-2,2]$, hence $$P(\lambda)=(\frac{\beta-\alpha}{4})^N\cdot {P}_N( \frac{4}{\beta-\alpha}\cdot (\lambda -\frac{\alpha+\beta}{2})).$$ Substitute $\beta$ into this formula, we have $P(\beta)=2(\frac{\beta-\alpha}{4})^N=2(\frac{r_0}{4})^N\cdot e^{iN\phi}.$ On the other hand, by Theorem \ref{main2}, $P(\beta)\in\{\pm 2Re^{i\varphi}\} $, so $R=(\frac{r_0}{4})^N$. By Proposition \ref{line spectrum}, we have $\varphi=N\phi-n\pi,n\in\Z$.
	
	$(3) \Rightarrow (1)$: For  
	$$ \cal{S}=\{\frac{\alpha+\beta}{2}+\frac{\beta-\alpha}{2}\cos(\frac{n\pi}{N})|n=1,2,\cdots,N-1\},$$ we have  $\#\cal{S}=N-1$.
	Through direct calculation, there are $P(\cal{S} \cup \{\alpha,\beta\})\subset \{\pm 2Re^{i\varphi}\}$. By Corollary \ref{N bands end to end}, $ \sigma(J)$ is composed of $N$ bands end to end, hence $ \sigma(J)$ is a curve connecting $\alpha$ and $\beta$.
\end{proof}

\section*{Acknowledgements}
J.You was partially supported by National Key R \&D Program of China (2020 YFA0713300) and Nankai Zhide Foundation.

\end{document}